\documentclass[11pt]{amsart}
\usepackage[dvipsnames]{xcolor}
\usepackage[utf8]{inputenc}
\usepackage{amssymb,amsmath,amsthm,amsfonts,fullpage,float,latexsym,bbm,microtype,cite}
\usepackage{mathrsfs}
\usepackage{array}
\usepackage{tikz}
\usetikzlibrary{decorations.pathreplacing}
\usepackage{hyperref}
\hypersetup{colorlinks=true, linkcolor=blue, citecolor=magenta, urlcolor=magenta}

\newtheorem{theorem}{Theorem}[section]
\newtheorem{proposition}[theorem]{Proposition}
\newtheorem{claim}[theorem]{Claim}
\newtheorem{conjecture}[theorem]{Conjecture}
\newtheorem{lemma}[theorem]{Lemma}
\newtheorem{corollary}[theorem]{Corollary}
\theoremstyle{definition}
\newtheorem{remark}[theorem]{Remark}
\newtheorem{definition}[theorem]{Definition}
\newtheorem{observation}[theorem]{Observation}
\newtheorem{example}[theorem]{Example}

\newtheorem{algorithm}{Algorithm}

\renewcommand{\emptyset}{\varnothing}

\newcommand{\R}{\mathbb{R}}
\newcommand{\Z}{\mathbb{Z}}

\newcommand{\bx}{\mathbf{x}}
\newcommand{\B}{\mathcal{B}}
\newcommand{\F}{\mathcal{F}}

\newcommand{\Po}{\mathcal{P}}

\DeclareMathOperator{\ehr}{\mathsf{ehr}}

\DeclareMathOperator{\Rel}{Rel} 

\newcommand{\defterm}[1]{\emph{\textbf{#1}}}
\DeclareMathOperator{\Pan}{\mathsf{Pan}} 

\DeclareMathOperator{\CF}{\mathsf{CF}} 

\DeclareMathOperator{\wt}{\mathsf{wt}}

\newcommand\commentout[1]{}

\makeatletter 
\newtheorem*{rep@theorem}{\rep@title}\newcommand{\newreptheorem}[2]{%
\newenvironment{rep#1}[1]{%
\def\rep@title{\bf #2 \ref{##1}}%
\begin{rep@theorem}}%
{\end{rep@theorem}}}
\makeatother
\newreptheorem{theorem}{Theorem}

\newtheorem*{rep@conjecture}{\rep@title}\newcommand{\newrepconjecture}[2]{%
\newenvironment{rep#1}[1]{%
\def\rep@title{\bf #2 \ref{##1}}%
\begin{rep@conjecture}}%
{\end{rep@conjecture}}}
\makeatother
\newreptheorem{conjecture}{Conjecture}

\title{Ehrhart Bounds for Panhandle and Paving Matroids Through Enumeration of Chain Forests}

\author{Danai Deligeorgaki}
\address{\scriptsize{KTH Royal Institute of Technology}, \url{https://www.kth.se/profile/danaide}}
\email{\scriptsize{danaide@kth.se }}

\author{Daniel McGinnis}
\address{\scriptsize{Department of Mathematics, Iowa State University}, \url{https://sites.google.com/iastate.edu/danielmcginnis}}
\email{\scriptsize{dam1@iastate.edu }}

\author{Andr\'es R. Vindas-Mel\'endez}
\address{\scriptsize{Departments of Mathematics, UC Berkeley \& Harvey Mudd College}, \url{https://math.berkeley.edu/~vindas}}
\email{\scriptsize{andres.vindas@berkeley.edu; arvm@hmc.edu}}

\date{}

\begin{document}

\begin{abstract}

Panhandle matroids are a specific family of lattice-path matroids corresponding to panhandle-shaped Ferrers diagrams.
Their matroid polytopes are the subpolytopes carved from a hypersimplex to form matroid polytopes of paving matroids. 
It has been an active area of research to determine which families of matroid polytopes are Ehrhart positive. 
We prove Ehrhart positivity for panhandle matroid polytopes, thus confirming a conjecture of Hanely, Martin, McGinnis, Miyata, Nasr, Vindas-Mel\'endez, and Yin (2023).
Another standing conjecture posed by Ferroni (2022) asserts that the coefficients of the Ehrhart polynomial of a connected matroid are bounded above by those of the corresponding uniform matroid.
We prove Ferroni's conjecture for paving matroids -- a class conjectured to asymptotically contain all matroids.
These results follow from purely enumerative  statements, the main one being conjectured by Hanely et. al concerning the enumeration of a certain class of ordered chain forests. 
Thus, the content and proofs in this paper reflect an enumerative combinatorics perspective, driven by considerations in Ehrhart theory.

\end{abstract}


\maketitle


\section{Introduction}

Two elemental objects in combinatorics and discrete geometry, namely matroids and polytopes, have garnered the attention of mathematicians for years. 
Matroid (base) polytopes form a bridge between these two objects. 
In this paper we resolve conjectures on the Ehrhart theory (the theory of integer points in polytopes) of two types of matroid polytopes: those arising from \emph{panhandle} and \emph{paving matroids}. 

A \defterm{matroid} $M$ on a finite ground set $E$ can be defined by its basis system $\B$, a nonempty collection of subsets of $E$, all of the same size, satisfying a certain \emph{exchange condition}. 
Given a matroid $M$ with ground set $[n]=\{1,2,\dots,n\}$ and rank~$r$, the \defterm{matroid (base) polytope} $\Po_M$ is the convex hull in $\R^n$ of the indicator vectors of its bases.
Matroid polytopes provide geometric insight into the structure of matroids.
A panhandle matroid, denoted $\Pan_{r,s,n}$, is a lattice-path matroid corresponding to panhandle-shaped Ferrers diagrams. It is described by the basis system:
\[
\B=\B_{r,s,n}=\left\{B\in\binom{[n]}{r}:\ |B\cap[s]|\geq r-1\right\}.
\]
See \cite{hanelyetal} for more information on panhandle matroids, although it is not needed to understand the main result of this paper and its proof.
A matroid of rank~$r$ is called \defterm{paving} if every \textit{circuit} has cardinality greater than or equal to~$r$, and \defterm{sparse paving} if it and its dual are both paving.
Paving matroids are noteworthy as it is conjectured and widely believed that asymptotically all matroids are paving, or even sparse paving (see \cite{mayhew}, \cite{pendavingh-vanderpol}, \cite[Chapter~15.5]{Oxley}).
For a more in depth background on matroids, we recommend ~\cite{Oxley} and~\cite{Welsh}.

The \defterm{Ehrhart function} of a polytope $P$ (i.e., the convex hull of finitely many points)  $P\subset \R^d$,  is the lattice-point enumerator $\ehr_P(t):=|tP\cap \Z^d|$, where $tP=\{t\bx:\ \bx\in P\}$ for $t\in \Z_{\geq 0}$. 
Ehrhart proved that when $P$ is a lattice polytope (i.e., its vertices have integer coordinates), the Ehrhart function is a polynomial in $t$ of degree equal to the dimension of $P$, commonly known as the \defterm{Ehrhart polynomial} \cite{Ehrhart}.  
The leading coefficient of the Ehrhart polynomial of a lattice polytope $P$ is equal to its relative volume, the second-leading coefficient is equal to half the relative surface area, and the constant coefficient is 1.
The other coefficients of $\ehr_P(t)$ are more mystifying, and in general can be negative (see \cite[Section 4]{liu} for several examples of Ehrhart polynomials with negative coefficients). 
Ehrhart polynomials are an important invariant of lattice polytopes, as they encode much of a polytope's geometry, arithmetic, and combinatorics. 
As a whole, Ehrhart theory has developed into a key topic, playing a role in and benefiting from other areas of mathematics including, but not limited to, number theory \cite{Beck, BaldoniVergne, MeszarosMorales}, commutative algebra \cite{CoxLittleSchenck, Fulton}, algebraic geometry \cite{Morelli, Pommersheim}, enumerative combinatorics \cite{beck-robins, beck-sanyal}, and integer programming \cite{Barvinok, Bruns23}.

In \cite{DeLoeraEhrhart2009}, De Loera, Haws, and K\"oppe conjectured that the Ehrhart polynomial of every matroid polytope has nonnegative coefficients (i.e., they are \defterm{Ehrhart positive}).
Their conjecture was disproved by Ferroni \cite{FerroniMatroids2022}, who constructed an infinite family of matroids whose matroid polytopes are not {Ehrhart positive}.
Determining families of matroid polytopes that do exhibit Ehrhart positivity is itself both a challenging and interesting problem, as evidenced by the extensive literature (see, for example \cite{hanelyetal,ferroni2023lattice,ferroniPositive, FerroniMinimal,fjs,jochemko-ravichandran,Knauer+}).
The techniques used to show Ehrhart positivity for different polytopes vary widely.
In this paper, we rely solely on combinatorial techniques.

Volume formulas for general matroid polytopes were given by Ardila, Benedetti, and Doker \cite{ArdilaBenedettiDoker} and Ashraf~\cite{Ashraf}, but their Ehrhart polynomials are not as well understood. 
One example of a family of matroid polytopes for which their Ehrhart polynomial is determined is \defterm{minimal matroids} (i.e., matroids with the least number of bases for their rank and ground set size among all such connected matroids).
In ~\cite{FerroniMinimal}, Ferroni calculated the Ehrhart polynomial of minimal matroid polytopes and proved that they are Ehrhart positive.
Subsequently, in \cite{FerroniMatroids2022}, Ferroni used the fact that the matroid polytopes of a sparse paving matroid could be obtained from a hypersimplex by slicing with (geometric) hyperplanes. 
Each piece sliced off in this way is itself a matroid polytope of the minimal matroid, which resulted in an explicit formula for the Ehrhart polynomial of a sparse paving matroid polytope, for which Ferroni was able to show is not always Ehrhart positive. 

In \cite{hanelyetal}, the authors extended Ferroni's methods from sparse paving to paving matroids and showed that the matroid polytope of any paving matroid can be obtained from slicing pieces off from the hypersimplex using hyperplanes.
The pieces sliced off are matroid polytopes of panhandle matroids.
With combinatorial techniques, they calculated the Ehrhart function of a panhandle matroid polytope and provided several equivalent formulas for it.
One such formula is the following: 

\begin{theorem}[Corollary 5.4 in \cite{hanelyetal}]\label{formula for positivity} Let $1\leq r\leq s< n$.
The Ehrhart polynomial of $\Po_{\Pan_{r,s,n}}$ can be written as
\[\ehr_{\Pan_{r,s,n}}(t)=
\frac{n-s}{(n-1)!}\binom{t+n-s}{n-s}
\varphi_{r,s,n}(t)\]
where
\begin{equation}\label{phisrlt}
\varphi_{r,s,n}(t)
=\sum_{i=0}^{s-r}(-1)^i\binom{s}{i}\sum_{\ell=0}^{s-1}(n-2-\ell)!\ell!\binom{s-1-\ell-i+t(s-r-i+1)}{s-1-\ell}\binom{s-1-i+t(s-r-i)}{\ell}.
\end{equation}
\end{theorem}

\noindent Though the authors presented the Ehrhart polynomial of a panhandle matroid\footnote{For simplicity, we write ``the Ehrhart polynomial of the matroid $M$" (denoted, $\ehr_M(t)$), to mean the ``Ehrhart polynomial of the matroid polytope of $M$".}, its Ehrhart positivity is not immediately evident.
Observe that, in the formula for $\ehr_{\Pan_{r,s,n}}(t)$ given in Theorem \ref{formula for positivity},
since $\binom{t+n-s}{n-s}$ is a polynomial in $t$ with positive coefficients, in order to prove the Ehrhart positivity of panhandle matroids, it suffices to show that the polynomial $\varphi_{r,s,n}(t)$ has positive coefficients.
In \cite{hanelyetal}, the authors conjectured that panhandle matroids are Ehrhart positive \cite[Conjecture 6.1]{hanelyetal}, and they showed that the positivity of $\varphi_{r,s,n}(t)$ follows from the purely combinatorial statement below. We defer the relevant definitions to Section \ref{sec:prelim}, and we include the statement here to emphasize that the contents of this paper are mostly of enumerative combinatorics. A proof of Conjecture \ref{conj:main} is our first main result.

\begin{conjecture}[Conjecture 6.9 in \cite{hanelyetal}]\label{conj:main}
Let $1\leq m\leq k\leq s$, $q\geq 0$, and $0\leq \ell\leq s-1$ be integers. The number of naturally ordered chain forests $\F$ of $[s]$ with $k$ blocks and weight $q$, such that $\gamma(\F,\ell) = m$, is given by 
\begin{equation}\label{eq:CF(q,n,k,l,m)}
|\CF(q,s,k,\ell,m)|=\sum_{i=0}^q(-1)^i\binom{s}{i} \Pi^{s-\ell-m}_{-i+1,s-1-\ell-i} \Pi^{\ell-(k-m)}_{s-\ell-i,s-1-i} \binom{k-1+q-i}{k-1}.
\end{equation}
\end{conjecture}
\addtocounter{theorem}{0}

The left-hand side is the number of combinatorial objects called \textit{ordered chain forests}, which are special ordered sets partitions with several additional properties, and the $\Pi$ terms on the right-hand side are sums of products of integers that can be thought of as elementary symmetric functions evaluated at integers. As mentioned earlier, the problem of proving Ehrhart positivity for panhandle matroids was reduced to Conjecture \ref{conj:main} in \cite{hanelyetal}. Thus, we obtain the following theorem. 

\begin{theorem}\label{thm:PosMain}
    Given positive integers $r\leq s\leq n-1$, the panhandle matroid $\Pan_{r,s,n}$ is Ehrhart positive.
\end{theorem}

\noindent In fact, we prove a stronger bound, namely that the Ehrhart polynomial of $\Pan_{r,s,n}$ is bounded below (coefficient-wise) by the Ehrhart polynomial of a product between a simplex and a hypersimplex (Remark \ref{rmk:better positivity}).

This work is in part motivated by the fact that panhandle matroids are special instances of other well-studied families of matroids, including Schubert matroids, lattice-path matroids, and positroids. 
Panhandle matroid polytopes are also instances of the more general family of (weighted) multi-hypersimplices, a family of alcoved polytopes considered in \cite{Lam2005AlcovedPI}. 
In fact, they are polypositroids \cite{lam-postnikov-polypositroids}, a subclass of alcoved polytopes that are conjectured to be Ehrhart positive \cite[Conjecture 6.3]{fjs}. 
Additionally, panhandle matroids are \textit{notched rectangle matroids}, defined in \cite{fan-li} (also known as \textit{cuspidal matroids} in \cite{ferroni2022valuative}), which are conjectured to be Ehrhart positive in both works. Hence, our results, namely Theorem \ref{thm:PosMain}, supports these conjectures. Furthermore, \textit{uniform matroids} are special instances of panhandle matroids as $\Pan_{r,n-1,n}=U_{r,n}$ is the uniform matroid with rank $r$ and ground set $[n]$. On the other hand, we have $\Pan_{r,r,n}$ is the minimal matroid $T_{r,n}$. Therefore, Theorem \ref{thm:PosMain} generalizes and interpolates between the corresponding Ehrhart positivity results of Ferroni showing that both $T_{r,n}$ and $U_{r,n}$ are Ehrhart positive \cite{ferroniPositive,FerroniMinimal}.

We then continue to discuss a related problem to Conjecture \ref{conj:main}. 
Ferroni posed in \cite{FerroniMinimal} a conjecture asserting that for a connected matroid $M$ with rank $r$ and $n$ elements, the Ehrhart polynomial of $\Po_M$ is coefficient-wise bounded below and above by the Ehrhart polynomial of the corresponding minimal matroid and the uniform matroid $U_{r,n}$, respectively.
In general, the lower bound is not true as Ferroni himself proved that there are matroids which fail to be Ehrhart positive \cite{FerroniMatroids2022} while the Ehrhart polynomials of minimal matroids have positive coefficients \cite{FerroniMinimal}. 
The second main result of this paper is to confirm Ferroni's Ehrhart-coefficient upper-bound conjecture, for the class of paving matroids.

\begin{theorem}\label{thm:UpperBound}
    Let $M$ be a paving matroid of rank $r$ with ground set $[n]$. Then
    \[
    \ehr_M(t) \preceq  \ehr_{U_{r,n}}(t).
    \]
\end{theorem}

\noindent The proof of this theorem relies on a general formula for the Ehrhart polynomial of a paving matroid that was first presented in \cite{hanelyetal}.
Again, a proof of Theorem \ref{thm:UpperBound} can be reduced to a combinatorial result akin to Conjecture \ref{conj:main}, which we resolve using similar methods.
We further show that the relaxation of an Ehrhart-positive matroid is also Ehrhart positive (see Section \ref{sec:upper_bound} for relevant definitions and results).

The techniques in this paper lay in the realm of enumerative combinatorics, and the proofs stand without particular knowledge of matroids, matroid polytopes, or Ehrhart theory. These topics are simply the sole motivation of this paper. Therefore, we omit detailed background information on these topics and refer to \cite{hanelyetal} for the precise definition of panhandle matroids and an overview of their Ehrhart theory.

The paper is structured as follows:
\begin{itemize}
    
    \item In Section \ref{sec:prelim}, we define \textit{ordered chain forests}, which are combinatorial objects that emerge from specific set partitions. We also define auxiliary combinatorial objects which are not directly present in the statement of Conjecture \ref{conj:main} but are required to support our proof of Conjecture \ref{conj:main}. 
    
    \item Section \ref{sec:phi} is concerned with understanding each term within the summation appearing in Conjecture \ref{conj:main}. 
    We do so by constructing a bijection between objects defined in Section \ref{sec:prelim}.

    \item Section \ref{sec:main_proof} presents the proof of Conjecture \ref{conj:main}, implying that panhandle matroids are Ehrhart positive, via a sign-reversing involution type argument.
    
    \item In Section \ref{sec:upper_bound}, we use the tools developed in previous sections to confirm Ferroni's Ehrhart-coefficient upper-bound conjecture in the case of paving matroids by proving another enumerative combinatorial statement akin to Conjecture \ref{conj:main}. We further show that Ehrhart-positivity is preserved through stressed-hyperplane relaxations.

\end{itemize}

\section{Toward Ehrhart positivity: Combinatorial objects  }\label{sec:prelim}

We recall that in Section 6 of \cite{hanelyetal}, the authors reduced the problem of proving Ehrhart positivity for panhandle matroids (Theorem \ref{thm:PosMain}) to a purely combinatorial problem (see Conjecture \ref{conj:main}).
We define \textit{ordered chain forests}, emerging as particular set partitions, which are intimately related to the conjecture. 
Specifically, we enumerate these objects whose cardinality can be expressed in terms of a weighted alternating sum over products involving the sum of all possible products of integers in a certain range, denoted $\Pi^n_{a,b}$. 
Additionally, we introduce \textit{valued ($A$-distinguished) ordered chain forests}, which will turn out to provide an interpretation of each term in the alternating sum.

We further note that (unordered) chain forests were presented in \cite{hanelyetal} and are used to state Conjecture \ref{conj:main}. 
The more general objects, which we refer to as ordered chain forests, are introduced to support our proof of the conjecture.

Since we introduce many definitions and notation in this section, we include a notation table for convenience.

\begin{table}[!ht]
\label{notation_table}
\centering
\caption{Notation Table}
\begin{tabular}{ | m{8em} | m{33em}|} 
\hline
\textbf{Symbol} & \textbf{Meaning} \\ \hline

 $\gamma(\F,\ell)$ & the number of trailers that appear after the $\ell^{th}$ position (when writing the sequence of numbers obtained by considering the elements as they appear in the ordered chain forest $\F$)\\
 \hline
 $\CF(q,s,k)$ ($\CF(q,s,k,\ell,m)$) & the set of naturally ordered chain forests of $[s]$ with $k$ blocks and weight $q$ (and $\gamma(\F,\ell)=m$) \\
 \hline
 $\Pi_{a,b}^n$ &  the sum of all possible products of $n$ integers lying between $a$ and $b$ \\
 \hline
 $B(\F)$ & the set of blocks of an ordered chain forest $\F$ \\
 \hline
 $v(\mathsf{B})$ & a nonnegative value assigned to a block $\mathsf{B}$  \\
  \hline
  $\wt(\mathsf{B})$ & the weight of a block $\mathsf{B}$\\
  \hline
  $DCF(q,s)$ ($DCF(q,s,k,\ell,m)$) & the set of $A$-distinguished ordered chain forests $(\F,v,A)$ of $[s]$ for some subset $A\subset [s]$ such that $|A| + \sum_{\mathsf{B}\in B(\F)} v(\mathsf{B}) = q$ (with $k$ blocks and $\gamma(\F,\ell) = q$) \\
  \hline
\end{tabular}
\end{table}

\begin{definition}
An \defterm{ordered chain forest} of $[s]$ is an ordered partition of $[s]$ into a set of blocks, where each block is itself internally ordered.\\

\noindent We write an ordered chain forest $\F$ as a sequence of blocks $\F=\mathsf{B}_1\cdots \mathsf{B}_k$.
Example \ref{ex:examples} illustrates some of these definitions.

\begin{enumerate}

\item The first and last element of a block is called its \defterm{leader} and \defterm{trailer}, respectively.

\item An ordered chain forest is said to be \defterm{naturally ordered} if the blocks are ordered increasingly according to their leaders. 

\item The \defterm{weight of a block} $\mathsf{B}$, denoted $\wt(\mathsf{B})$, in an ordered chain forest is the number of elements of the block that are less than its leader. 
If an element is less than the leader of its block, we say this element \defterm{contributes weight} to the block.

\item The \defterm{weight of an ordered chain forest} is the sum of the weights of its blocks.
\end{enumerate}

\begin{example}\label{ex:examples}
Consider the ordered chain forests $\F_1=[1,3][4,2][5]$ and $\F_2=[4,2][1,3][5]$.
Observe that $\F_1$ is naturally ordered, while $\F_2$ is not, and they both are distinct from one another because the ordering of their blocks differs. 
Additionally, $[3,1][4,2][5]$ is distinct from $[1,3][4,2][5]$ as well since the ordering of the elements within each block matters.  
Now, consider the ordered chain forest $\F=[1,3][5,2,4][7,6][8].$ 
The weights of each block are 0, 2, 1, and 0, respectively.
The total weight of $\F$ is 3.
\end{example}
\end{definition}

Given an ordered chain forest $\F=\mathsf{B}_1\cdots \mathsf{B}_k$ of $[s]$, for $0\leq l \leq s-1$, we define
\[
\gamma(\F,\ell): = \left|\left\{j \in [k] \mid \sum_{i=1}^j |\mathsf{B}_i| >\ell\right\}\right|.
\]
In other words, if we make note of the sequence of numbers obtained by writing the elements as they appear in $\F$, then $\gamma(\F,\ell)$ is the number of trailers that appear after the $\ell^{th}$ position.
For example, if $\F=[1,3][4,2][5]$, then the corresponding sequence of numbers is $1\ 3\ 4\ 2\ 5$ and the trailers are 3, 2, and 5.
Here, we say 1 is in the first position, 3 is in the second position, 4 is in the third position, and so on. Thus, we have that 
\[
\gamma(\F,0) = 3,\, \gamma(\F,1) = 3,\, \gamma(\F,2) = 2,\, \gamma(\F,3) = 2,\, \gamma(\F,4) = 1. \]
\smallskip

\begin{definition}
    For integers $q\geq 0$ and $1\leq k\leq s$, we define $\CF(q,s,k)$ to be the set of naturally ordered chain forests of $[s]$ with $k$ blocks and weight $q$.
\end{definition}

The set $\CF(q,s,k)$ was defined by Ferroni \cite{ferroniPositive} (although with different notation), and revisited in \cite{hanelyetal} where it is also denoted by $\CF(q,s,k)$.
In both papers, a certain combinatorial expression was shown to enumerate the quantity $|\CF(q,s,k)|$.
In order to present this particular result, we introduce the following notation.

\begin{definition}
    Let $a\leq b$ and $n$ be integers. We define the quantity $\Pi_{a,b}^n$ as 
    \[
    \Pi_{a,b}^n = \sum_{a\leq i_1<\cdots<i_n \leq b} i_1\cdots i_n,
    \]
    i.e., as the sum of all possible products of $n$ integers between $a$ and $b$.
    We define $\Pi^0_{a,b} = 1$ and $\Pi^n_{a,b} = 0$ if $n<0$ or $n > b-a+1$.
    For certain choices of parameters in this paper, we sometimes encounter terms of the form $\Pi_{a+1,a}^n$. In this special case, we use the convention that $\Pi_{a+1,a}^n = 1$.
\end{definition}

\noindent This notation now allows us to enumerate the set of naturally ordered chain forests. 

\begin{theorem}[\cite{ferroniPositive,hanelyetal}]\label{lah theorem}
Let $q\geq 0$, $1\leq k\leq s$ be integers. The quantity $|\CF(q,s,k)|$ is given by 
\begin{equation}\label{eq:Lah}
|\CF(q,s,k)|=\sum_{i=0}^q(-1)^i\binom{s}{i} \Pi^{s-k}_{-i+1,s-1-i} \binom{k-1+q-i}{k-1}.
\end{equation}
\end{theorem}

\noindent This was shown through a generating function approach in \cite{ferroniPositive} and later by combinatorial means in \cite{hanelyetal}.
Additionally, a refinement of the set $\CF(q,s,k)$ was defined in \cite{hanelyetal} for the purpose of articulating Conjecture \ref{conj:main}, which is  a generalization of Theorem \ref{lah theorem}.

\begin{definition} 
Let $1\leq m\leq k\leq s$, $q\geq 0$, and $0\leq \ell\leq s-1$ be integers.
We denote the set of naturally ordered chain forests $\F$ of $[s]$ with $k$ blocks with $\gamma(\F,\ell) = m$ and weight $q$ by $\CF(q,s,k,\ell,m)$.
\end{definition} 

We restate Conjecture \ref{conj:main} below showing that a certain expression, which arose from an analysis of the Ehrhart polynomials of panhandle matroids in \cite{hanelyetal}, enumerates the quantity $|\CF(q,s,k,\ell,m)|$.

{
\renewcommand{\thetheorem}{\ref{conj:main}}
\begin{conjecture}[Conjecture 6.9 in \cite{hanelyetal}]
Let $1\leq m\leq k\leq s$, $q\geq 0$, and $0\leq \ell\leq s-1$ be integers. The number of naturally ordered chain forests $\F$ of $[s]$ with $k$ blocks and weight $q$, such that $\gamma(\F,\ell) = m$, is given by 
\begin{equation}
|\CF(q,s,k,\ell,m)|=\sum_{i=0}^q(-1)^i\binom{s}{i} \Pi^{s-\ell-m}_{-i+1,s-1-\ell-i} \Pi^{\ell-(k-m)}_{s-\ell-i,s-1-i} \binom{k-1+q-i}{k-1}.
\end{equation}
\end{conjecture}
}

We again emphasize that the contents of Section 6 in \cite{hanelyetal} show that the truth of Conjecture \ref{conj:main} implies Theorem \ref{thm:PosMain}. 
In Section \ref{sec:main_proof}, we prove Conjecture \ref{conj:main}, thus settling the Ehrhart positivity conjecture for panhandle matroids.
Although the expression appearing in Equation (\ref{eq:CF(q,n,k,l,m)}) is similar to Equation (\ref{eq:Lah}), and different proofs for Equation (\ref{eq:Lah}) are given in \cite{ferroniPositive} and \cite{hanelyetal}, the methods of these two proofs do not seem to extend directly to prove Conjecture \ref{conj:main}.
The proof we present here is purely combinatorial, and hence more in line with the ideas in \cite{hanelyetal}.
However, our approach is novel and our proof is substantially different.
Our methods use only combinatorial techniques, using auxiliary combinatorial objects and sophisticated manipulations. 
Our proofs can be viewed as another example of counting with signs (see for instance section 2.2 in \cite{sagan2020art}), although in our case the general argument is rather involved.
Our approach relies on the following definitions, which are auxiliary combinatorial objects that we will need to enumerate $\CF(q,s,k,\ell,m)$.

\begin{definition}
For an ordered chain forest $\F=\mathsf{B}_1\cdots \mathsf{B}_k$, let $B(\F)$ be the set of blocks of $\F$.    
\begin{enumerate}

\item A \defterm{valued ordered chain forest} is a pair $(\F,v)$ where $\F$ is an ordered chain forest and $v:B(\F) \longrightarrow \mathbb{Z}_{\geq 0}$ is a function assigning a nonnegative integer value to each block.

\item For a block $\mathsf{B}\in B((\F,v))$ of a valued ordered chain forest, we call $v(\mathsf{B})$ the \defterm{value} of $\mathsf{B}$.

\item The \defterm{weight of a valued ordered chain forest} $(\F,v)$ is the sum of weights of blocks that make up $(\F,v)$ and $\sum_{\mathsf{B}\in B((\F,v))} v(\mathsf{B})$.

\item At times we denote a valued ordered chain forest $(\F,v)$ as $(\F,v)=\mathsf{B}_1^{v(\mathsf{B}_1)}\cdots \mathsf{B}_k^{v(\mathsf{B}_k)}$.

\end{enumerate}

\end{definition}

\begin{definition}\label{A-dist forests}

Let $A\subset [s]$.
An \defterm{$A$-distinguished ordered chain forest} of $[s]$ is a pair $(\F,A)$ where $\F$ is an ordered chain forest, $\F=\mathsf{B}_1\cdots \mathsf{B}_r\mathsf{C}_1\cdots \mathsf{C}_p$, where $0\leq r,p \leq s$, with the following properties:

    \begin{itemize}
    
        \item the blocks $\mathsf{B}_i$ consist only of elements in $[s]\setminus A$,
        
        \item the blocks $\mathsf{C}_i$ consist only of elements in $A$,
        
        \item each block has weight 0, i.e., the leader of each block is also the smallest element of the block,
        
        \item the blocks $\mathsf{B}_1\cdots \mathsf{B}_r$ are ordered increasingly according to their leader, and 
        
        \item the blocks $\mathsf{C}_1\cdots \mathsf{C}_p$ are ordered decreasingly according to their leader.
        
    \end{itemize}

\noindent A \defterm{valued $A$-distinguished ordered chain forest} of $[s]$ is a triple $(\F,v,A)$, where $(\F,A)$ is an $A$-distinguished ordered chain forest and $v$ is a function $v:B(\F) \longrightarrow \mathbb{Z}_{\geq 0}$, assigning nonnegative integer values to the blocks $\mathsf{B}_1,...,\mathsf{B}_r, \mathsf{C}_1,...,\mathsf{C}_p$ of $\F$.
We may also write $(\F,v,A)$ as $(\F,v,A)=(\mathsf{B}_1^{v(\mathsf{B}_1)}\cdots \mathsf{B}_r^{v(\mathsf{B}_r)}\mathsf{C}_1^{v(\mathsf{C}_1)}\cdots \mathsf{C}_p^{v(\mathsf{C}_p)},A)$.
Additionally, we denote by $B_A(\F)$ the set of blocks consisting only of elements in $A$, i.e., $B_A(\F) = \{\mathsf{C}_1,\dots,\mathsf{C}_p\}$.
Note that $B_\emptyset(\F)=\emptyset$.

\end{definition}

We can think of the values on the blocks as adding a different kind of weight to the blocks, one that does not rely on the internal ordering of the block. While it may not be immediately clear why this notion of value is relevant since it does not appear in the definition of the elements in $\CF(q,s,k,\ell,m)$, we will see that this is a necessary auxiliary component that is needed to form a connection between the left-hand and right-hand side of Conjecture \ref{conj:main}. Hence, it can be thought of as a kind of ``auxiliary weight''. Proposition \ref{prop:i-interpretation} below shows the connection between valued $A$-distinguished ordered chain forests and the right-hand side, while the contents of Section \ref{sec:phi} forms the connection with the left-hand side.

\begin{example}   
Observe that $([1,4]^2[3][5,6]^1[2],\{2\})$ is a valued $\{2\}$-distinguished ordered chain forest, and that $([1,4]^2[3][5,6]^1[2],\{2,5,6\})$ is a valued $\{2,5,6\}$-distinguished ordered chain forest.
\end{example}

\begin{definition}
For $q\geq 0$ and $s\geq 1$, let $DCF(q,s)$ be the set of {valued} $A$-distinguished ordered chain forests $(\F,v,A)$ of $[s]$ for some $A\subset [s]$, such that
\[
|A| + \sum_{\mathsf{B}\in B(\F)} v(\mathsf{B}) = q.
\]
We further denote by $DCF(q,s,k,\ell,m)$ the set of {valued $A$-distinguished ordered chain forests} $(\F,v,A)\in DCF(q,s)$ with the additional property that $\F$ has $k$ blocks and $\gamma(\F,\ell) = m$, where $1\leq m\leq k\leq s$, $0\leq \ell\leq s-1$.
\end{definition}

We introduce the sets $DCF(q,s)$ and $DCF(q,s,k,\ell,m)$ as a means to the end of enumerating $\CF(q,s,k,\ell,m)$. Again these are auxiliary sets that are necessary for our proofs.

The following proposition provides a combinatorial interpretation for each summand in Conjecture \ref{conj:main}. 
The purpose of introducing {(valued)} $A$-distinguished ordered chain forests is to properly articulate Proposition \ref{prop:i-interpretation}. 
This proposition allows us to view the expression from Conjecture \ref{conj:main} through a combinatorial lens that will support our proof. Indeed, the left-hand side of equation \ref{sum interpretation} appears as a summand in Conjecture \ref{conj:main}.

\begin{proposition}\label{prop:i-interpretation}
For an integer $0\leq i\leq q$ and fixed
$1\leq m\leq k \leq s$, $0\leq \ell\leq s-1$, the $i^{\text{th}}$ term of the sum in Conjecture \ref{conj:main} can be expressed in terms of valued $A$-distinguished ordered chain forests as follows:

\begin{equation}\label{sum interpretation}
(-1)^i\binom{s}{i} \Pi^{s-\ell-m}_{-i+1,s-\ell-1-i} \Pi^{\ell-(k-m)}_{s-\ell-i,s-1-i} \binom{k-1+q-i}{k-1} = \sum_{A\in \binom{[s]}{i}}\sum_{(\F,v,A)\in DCF(q,s,k,\ell,m)}(-1)^{|B_A(\F)|}.
\end{equation}
\end{proposition}

\begin{proof}
Let us start by looking at the left hand side of \eqref{sum interpretation}, which equals

\begin{align}\label{insignificant equation}
\begin{split}
&(-1)^i\binom{s}{i} \Pi^{s-\ell-m}_{-i+1,s-\ell-1-i} \Pi^{\ell-(k-m)}_{s-\ell-i,s-1-i} \binom{k-1+q-i}{k-1}\\
=& \sum_{A \in \binom{[s]}{i}}(-1)^i\Pi^{s-\ell-m}_{-i+1,s-\ell-1-i} \Pi^{\ell-(k-m)}_{s-\ell-i,s-1-i} \binom{k-1+q-i}{k-1}.
\end{split}
\end{align}
Note that $\Pi^{s-\ell-m}_{-i+1,s-\ell-1-i} \Pi^{\ell-(k-m)}_{s-\ell-i,s-1-i}$ is the sum of all {possible} products of the form \[i_1\cdots i_{s-\ell-m}\cdot j_1\cdots j_{\ell-(k-m)},\] where \[-i+1\leq i_1 < \cdots < i_{s-\ell -m} \leq s-\ell -1 -i\] and \[s-\ell-i\leq j_1 < \cdots < j_{\ell -(k-m)} \leq s -1 -i.\] 

The following claim shows how the product $i_1\cdots i_{s-\ell-m}\cdot j_1\cdots j_{\ell-(k-m)}$ corresponds to $|i_1|\cdots |i_{s-\ell-m}|\cdot |j_1|\cdots |j_{\ell-(k-m)}|$ distinct  $A$-distinguished ordered chain forests when the $i$-element subset $A$ is fixed. For an $A$-distinguished ordered chain forest $(\F,A)$, let $B_A$ be the number of blocks consisting only of elements in $A$.
\begin{claim}
Let $i_1\cdots i_{s-\ell-m}\cdot j_1\cdots j_{\ell-(k-m)}$ be a product as above and fix an $i$-element subset $A\in \binom{[s]}{i}$. Let $1\leq u_1<\cdots<u_{k-1}\leq s-1$  be the integers such that $s-i-u_r$ does not appear in the product for all $1\leq r\leq k-1$. There are exactly $|i_1|\cdots |i_{s-\ell-m}|\cdot |j_1|\cdots |j_{\ell-(k-m)}|$
distinct $A$-distinguished ordered chain forests whose blocks start at positions $1,u_1+1,u_2+1,\dots,u_{k-1}+1$.

Furthermore, the sum of values $(-1)^{i-B_A}$ over all the aforementioned $A$-distinguished ordered chain forests above is equal to $i_1\cdots i_{s-\ell-m}\cdot j_1\cdots j_{\ell-(k-m)}$.
\end{claim}
\begin{proof}
Denote by $S= \{ i_1,\dots, i_{s-\ell-m}, j_1,\dots, j_{\ell-(k-m)}\}$ the set of elements in the product.
For $1\leq u\leq s-1$, let $e_u = s-i-u$ be the $u^{th}$ largest element among $-i+1,\dots, s-1-i$.
We define a sequence $(a_0,a_1,\dots,a_{s-1})$ consisting of numbers and left brackets ``$[\phantom{s}$'' as follows: 

\begin{itemize}

\item $a_0 = [\phantom{s}$
    
\item $a_u = e_u$ if $e_u\in S$

\item $a_u=[\phantom{s}$ if $e_u\notin S$.

\end{itemize}

We now show how the sequence $(a_0,a_1,\dots,a_{s-1})$ can be used to construct a set of 
\[|i_1|\cdots |i_{s-\ell-m}|\cdot |j_1|\cdots |j_{\ell-(k-m)}| \]
distinct $A$-distinguished ordered chain forests whose starting block positions are $1,u_1+1,u_2+1,\dots,u_{k-1}+1$. 

If we have $a_u = [\phantom{s}$ for some $u$, this corresponds to starting a block at the $(u+1)^{th}$ position.
In particular, this is why $a_0=[\phantom{s}${; to ensure that a block always starts at the first position}.
Starting at $u=0$, we construct a family of $A$-distinguished chain forests as follows.

\begin{itemize}

\item If $u\leq s-1-i$ and $a_u=[\phantom{s}$, then we start a new block at the $(u+1)^{th}$ position and we fill this position with the smallest element of $[s]\setminus A$ that has not yet been allocated to a block. 
In particular, since $a_0=[ \phantom{s}$, we fill the first position with the smallest element of $[s]\setminus A$. 
Note that when we start a new block at position $(u+1)$, we end the previous block at position $u$.
    
\item If $u\leq s-1-i$ and $a_u\neq [\phantom{s}$, we can place any element from $[s]\setminus A$ that has not been already allocated to a block at the $(u+1)^{th}$ position, for which there are $a_u$ choices.

\end{itemize}

This process constructs the blocks containing elements of $[s]\setminus A$. 
Let us now construct the blocks containing elements of $A$ from right to left.
We consider the largest index $u_1$ with $u_1\geq s-i$ such that $a_{u_1}=[\phantom{s}$, and start a new block at the $(u_1+1)^{th}$ position.
We fill this position with the smallest element of $A$.
Then we may fill in the rest of this block in $|a_{u_1+1}|\cdots |a_{s-1}|$ ways with elements of $A$ that were not yet chosen. 
Each such way to fill in the block contributes $(-1)^{s-1-u_1}$ to the sign of the product {$i_1\cdots i_{s-\ell-m}\cdot j_1\cdots j_{\ell-(k-m)}$} since $a_u< 0$ for $u\geq s-i+1$. 
Then we consider the second-largest $u_2< u_1$ with $u_2\geq s-i$ such that $a_{u_2}=[\phantom{s}$.
We start a new block at the $(u_2+1)^{th}$ position and we fill this position with the smallest element of $A$ not yet chosen. Then we may fill in the rest of this block in $|a_{u_2+1}\cdots a_{u_1-1}|$ ways with elements of $A$ that were not yet chosen. 
Similarly, each such choice contributes $(-1)^{u_1-u_2-1}$ to the product.
We continue in this way until we have allocated all elements of $A$.

If the product $i_1\cdots i_{s-\ell-m}\cdot j_1\cdots j_{\ell-(k-m)}$ is not zero, then $0\notin S$, so we always start a new block at the $(s-i+1)^{th}$ position and it is filled with an element from $A$. 
Therefore, the last $i$ positions are filled with elements of $A$ and each block in the ordered chain forests constructed consists either entirely of elements in $A$ or entirely of elements in $[s]\setminus A$. 
The remaining conditions in Definition \ref{A-dist forests} are also satisfied, i.e., the ordered chain forests constructed are $A$-distinguished.

Additionally, as discussed a ``$-1$" is contributed to the sign of the product $i_1\cdots i_{s-\ell-m}\cdot j_1\cdots j_{\ell-(k-m)}$ every time we choose an element in $A$ that is not the leader of a block, and in any other case that we choose some element, it does not affect the sign of the product. 
Therefore, if $B_A$ is the number of blocks that consist only of elements in $A$, then the sign of $i_1\cdots i_{s-\ell-m}\cdot j_1\cdots j_{\ell-(k-m)}$ is $(-1)^{i-B_A}$ since ${|A| = i }$.
\end{proof}

The $A$-distinguished ordered chain forests $(\F,A)$ constructed in this way all have $k$ blocks since there are $k-1$ elements between $-i-1$ and $s-1-i$ not contained in the product.
To see that $\gamma(\F,\ell) = m$, notice that in the product  $i_1\cdots i_{s-\ell-m}\cdot j_1\cdots j_{\ell-(k-m)}$ we choose $\ell - (k-m)$ elements out of the $\ell$ elements between $s-\ell - i$ and $s-1-i$. 
Hence, among the first $\ell + 1$ positions, there are exactly $k-m+1$ positions that are the beginning of some block (including the first position).
Therefore, there are exactly $k-m$ blocks that end on or before the $\ell^{th}$ position, which implies that there are exactly $m$ blocks that end after the $\ell^{th}$ position.

It then follows that for a fixed set $A\in \binom{[s]}{i}$, 
\[
(-1)^i\Pi^{s-\ell-m}_{-i+1,s-\ell-1-i} \Pi^{\ell-(k-m)}_{s-\ell-i,s-1-i}
\] 
is given by $\sum (-1)^{|B_A(\F)|}$ where the sum is taken over all $A$-distinguished ordered chain forests $(\F,A)$ of $[s]$ where $\F$ has $k$ blocks and $\gamma(\F,\ell) = m$.

Finally, the quantity $\binom{k-1+q-i}{k-1}$ can be interpreted as the number of ways to assign nonnegative values to the $k$ blocks of such $A$-distinguished ordered chain forests $(\F,A)$ whose sum of the values on the blocks is $q-i$. 

Since $q-i +|A| = q$, we have that 
    \[
    (-1)^i\Pi^{s-\ell-m}_{-i+1,s-\ell-1-i} \Pi^{\ell-(k-m)}_{s-\ell-i,s-1-i}\binom{k-1+q-i}{k-1} = \sum_{(\F,v,A)\in DCF(q,s,k,\ell,m)}(-1)^{|B_A(\F)|}.
    \] 
Summing over all subsets $A \in \binom{[s]}{i}$ completes the proof, by Equation \eqref{insignificant equation}.
\end{proof}

The following example demonstrates the process of creating $A$-distinguished ordered chain forests starting from a product, as discussed in Proposition \ref{prop:i-interpretation}.

\begin{example}
Let $s=9$, $k=4$, $m=2$, $i=4$, $\ell= 5$ and take $A=\{1,2,3,4\}$.
Let $(-3)\cdot (-1)\cdot 1\cdot 3\cdot 4$ be a product contributing to the sum defining $\Pi^2_{-3,-1}\Pi^3_{0,4}$.
The associated sequence as in the proof of Proposition \ref{prop:i-interpretation} is
 \[
    \text{{$(a_0,...,a_8)=$}}([\ ,\,4,\,3,\,[\ ,\,1,\,[\ ,\,-1,\,[\ ,\,-3).
\]
Thus, the blocks of any $A$-distinguished ordered chain forest corresponding to this product have the following form:
    \[
    [\_,\_,\_][\_,\_][\_,\_][\_,\_].
    \]
The first position will be filled with $5$, the smallest element of $[9]\setminus A$, then there are $4\cdot 3$ ways to fill in the rest of this block with elements from $[9]\setminus (A\cup \{5\})=\{6,7,8,9\}$.
The smallest element of $[9]\setminus A$ not yet used will fill in the fourth position, and then there is 1 way to fill in the rest of the second block with that last element of $[9]\setminus A$ not yet chosen.

The eighth position is then filled with $1$, the smallest element of $A$, and there are 3 ways to fill in the rest of that block with an element of $A\setminus\{1\}$: specifically, the last block can be chosen among $[1,2], [1,3], [1,4]$. 
Then we fill in the sixth position with the smallest element of $A$ not yet used. 
Finally, there is only 1 way to fill the seventh position with the last element of $A$.

Each $A$-distinguished ordered chain forest obtained in this way contributes $(-1)^2=1$ to the sum $\Pi^2_{-3,-1}\Pi^3_{0,4}$ since there are two blocks consisting only of elements in $A$.
\end{example}


\section{A bijection between ordered chain forests}\label{sec:phi}

Recall that the blocks of a valued $A$-distinguished ordered chain forest $(\F,v,A)$ have weight 0 by definition {(see Definition \ref{A-dist forests})}. 
In what follows, we define an algorithm which takes as input such chain forests $(\F,v,A)$ and outputs a valued ordered chain forest where some blocks may have non-zero weight. 
We define a map $\phi$ via this algorithm, that will play a crucial role throughout the proofs.
We will see each element in $\CF(q,s,k,\ell,m)$ is in the image of some valued $\emptyset$-distinguished ordered chain forest under $\phi$. 
Furthermore, $\phi$ will allow us to relate the combinatorial interpretation from Proposition \ref{prop:i-interpretation} to the statement of Conjecture \ref{conj:main}.\\

\subsection{An algorithm defined on valued \texorpdfstring{$A$}{}-distinguished ordered chain forests.}

Let $(\F,v,A) = (\mathsf{B}_1^{\beta_1}\cdots \mathsf{B}_r^{\beta_r} \mathsf{C}_1^{\gamma_1}\cdots \mathsf{C}_p^{\gamma_p},A)$ be a valued $A$-distinguished ordered chain forest of $[s]$ where the blocks $\mathsf{B}_i$ are the blocks with elements in $[s]\setminus A$ and the blocks $\mathsf{C}_i$ are the blocks with elements in $A$ {$(0\leq r,p\leq s)$}. 
We first split $(\F,v,A)$ into its \defterm{non-distinguished part}, $(\F_1,v_1) = \mathsf{B}_1^{\beta_1}\cdots \mathsf{B}_r^{\beta_r}$, and its \defterm{distinguished part}, $(\F_2,v_2)=\mathsf{C}_1^{\gamma_1}\cdots \mathsf{C}_p^{\gamma_p}$. 

We now define the algorithm on the the non-distinguished part $(\F_1,v_1)$. 
We initiate with a set $P=\emptyset$, which we call the set of \textit{processed elements}, and an ordered list $L$ that contains the elements in $[s]\setminus A$ listed increasingly. 
At each iterative step in the algorithm, the set $P$ and the ordered list $L$ will be updated. 
Moreover, at each iteration, $L$
will represent an \textit{ordering} on the elements of $[s]\setminus A$. 
We represent this ordering with the notation $<_L$. 
We say that $a$ is smaller (resp. larger) than $b$ according to $L$ if $a<_Lb$ (resp. $a>_Lb$). 
If we say $a$ is smaller (resp. larger) than $b$ without specifying $L$, then we mean that $a<b$ (resp. $a>b$) in the natural ordering. 

\begin{algorithm}[Processing valued ordered chain forests]\label{alg}\hfill \\
{$\mathtt{Input}$: valued ordered chain forest $(\F_1,v_1)$, $P=\emptyset$, naturally-ordered list $L$ on $[s]\setminus A$.}\\
$\mathtt{Output}$: valued ordered chain forest $(\F'_1,v'_1)$ obtained at termination of the algorithm whose properties are described in Proposition \ref{prop:phi-image}.
\begin{enumerate}

    \item Find the leftmost block $\mathsf{B}\in B(\F_1)$ for which $\mathsf{B}$ contains an element not in $P$ and $v_1(\mathsf{B}) > 0$.
    Let $a$ be the leading term of $\mathsf{B}$, and let $i_1<_L i_2<_L \cdots<_L i_j$ be the elements lying either in $\mathsf{B}$ or in one of the blocks to the right of $\mathsf{B}$ (notice that $i_1=a$). 
    
    \item Update the blocks in $\F_1$ by applying the map $i_x \mapsto i_{x+1}$ (indices are modulo $j$) and fixing the remaining elements.
    
    \item $L=L+[a]$, i.e., append $a$ to the end of $L$.
    \item $P=P\cup \{a\}$.
    We say $a$ has been ``\emph{processed}.''
    
    \item $v_1(\mathsf{B})=v_1(\mathsf{B})-1$.
    
    \item If there exists a block that contains an element not in $P$ whose value is positive, go back to step (1), otherwise terminate.
\end{enumerate}

\end{algorithm}

\noindent Through this algorithm, we define the following map on the set of valued $A$-distinguished ordered chain forests.

\begin{definition}
Let $(\F,v,A) = (\mathsf{B}_1^{\beta_1}\cdots \mathsf{B}_r^{\beta_r}\mathsf{C}_1^{\gamma_1}\cdots \mathsf{C}_p^{\gamma_p}\text{{$,A$)}}$ be an element of $DCF(q,s)$ with  non-distinguished part $(\F_1,v_1)=\mathsf{B}_1^{\beta_1}\cdots \mathsf{B}_r^{\beta_r}$. 
Let $(\F_1',v_1') = \mathsf{B}_1'^{\beta_1'}\cdots \mathsf{B}_r'^{\beta_r'}$ be the {output of Algorithm \ref{alg}}.
We define the \defterm{map $\phi$} by
\[
\phi((\F,v,A)) = (\mathsf{B}_1'^{\beta_1'}\cdots \mathsf{B}_r'^{\beta_r'}\mathsf{C}_1^{\gamma_1}\cdots \mathsf{C}_p^{\gamma_p},A).
\]
\end{definition}

Note that an element in the image of $\phi$ is not necessarily an $A$-distinguished chain forest since the leaders of some blocks may not be the smallest element in the block (see the examples below). However, it is still necessary to view an element in the image as a triple, i.e. distinguishing the elements of $A$ is crucial to our proofs later on. The basic idea of the algorithm is it converts that values on the blocks to weight on the blocks and hence the sum of the values and weights on the blocks is preserved.

\begin{example}
Let $(\F,v,\emptyset)= ([1,6,2]^2[3,7,5]^1[4],\emptyset)$ be a valued $\emptyset$-distinguished ordered chain forest. 
Then, $(\F_1,v_1)\ =\ [1,6,2]^2[3,7,5]^1[4]$ is itself the undistinguished part, and {the valued ordered chain forest, together with the sets $P$ and $L$, after} each iterative step  in Algorithm \ref{alg} is described below:

\begin{itemize}
\item $[1,6,2]^2[3,7,5]^1[4]$, $P=\emptyset$, $L=[1,2,3,4,5,6,7]$

\item $[2,7,3]^1[4,1,6]^1[5]$, $P=\{1\}$, $L=[2,3,4,5,6,7,1]$

\item $[3,1,4][5,2,7]^1[6]$, $P=\{1,2\}$, $L=[3,4,5,6,7,1,2]$

\item $[3,1,4][6,5,2][7]$, $P=\{1,2,5\}$, $L=[3,4,6,7,1,2,5]$

\end{itemize}

Thus, {$(\F'_1,v'_1)=[3,1,4][6,5,2][7]$ and} $\phi((\F,v,\emptyset)) = ([3,1,4][6,5,2][7],\emptyset)$.
\end{example}

\begin{example}\label{ex:68}
Let $(\F,v,\{6,8\}) = ([1,5,3]^2[2]^2[4,7]^1[8][6]^1,\{6,8\})$ be a valued $\{6,8\}$-distinguished ordered chain forest. 
The undistinguished part is given by $(\F_1,v_1) = [1,5,3]^2[2]^2[4,7]^1$.
{The valued ordered chain forest, together with the sets $P$ and $L$, after} each iterative step in Algorithm \ref{alg} is described below:

\begin{itemize}

\item $[1,5,3]^2[2]^2[4,7]^1$, $P=\emptyset$, $L=[1,2,3,4,5,7]$

\item $[2,7,4]^1[3]^2[5,1]^1$, $P=\{1\}$, $L=[2,3,4,5,7,1]$

\item $[3,1,5][4]^2[7,2]^1$, $P=\{1,2\}$, $L=[3,4,5,7,1,2]$

\item $[3,1,5][7]^1[2,4]^1$, $P=\{1,2,4\}$, $L=[3,5,7,1,2,4]$

\item $[3,1,5][2][4,7]^1$, $P=\{1,2,4,7\}$, $L=[3,5,1,2,4,7]$

\end{itemize}

Thus, {$(\F'_1,v'_1)=[3,1,5][2][4,7]^1$ and}  $\phi((\F,v,\{6,8\})) = ([3,1,5][2][4,7]^1[8][6]^1,\{6,8\})$. 
\end{example}

\subsection{Properties of the map \texorpdfstring{$\phi$}{}}

In the following, we note a few facts pertaining to Algorithm \ref{alg} that 
{will be used to understand $\phi$ and prove that it is a bijection (Propositions \ref{prop:phi-bijection} and \ref{prop:phi-image}). 
These will serve as the foundations for proving Conjecture \ref{conj:main}.}

\begin{lemma}\label{lem:BlockOrder}
Let $(\F,v,A)\in DCF(q,s)$, and let $(\F_1,v_1)$ be its non-distinguished part. 
After the $i^{th}$ iteration of {the steps in} Algorithm \ref{alg} {with input} $(\F_1,v_1)$, the blocks of $\F_1$ are ordered increasingly by their leaders according to $L$, and the leader of each block is the smallest element of that block according to $L$.
\end{lemma}

\begin{proof}
We proceed by induction.
Let $\mathsf{B}_{1,i}^{\beta_{1,i}}\cdots {\mathsf{B}_{k,i}^{\beta_{k,i}}}$ denote the valued ordered chain forest and $L_i$ be the ordered list {$L$} obtained after the $i^{th}$ iteration of {the steps in} Algorithm \ref{alg} {with input} $(\F_1,v_1)$.
By definition, the statement of the lemma is true for $(\F_1,v_1) = \mathsf{B}_{1,0}^{\beta_{1,0}}\cdots {\mathsf{B}_{k,0}^{\beta_{k,0}}}$.
Assume that the statement is true for  $\mathsf{B}_{1,i}^{\beta_{1,i}}\cdots {\mathsf{B}_{k,i}^{\beta_{k,i}}}$. 
Let $j$ be the smallest index such that $\mathsf{B}_{j,i}$ contains a non-processed element and $\beta_{j,i} > 0$, and let $a$ be the leader of $\mathsf{B}_{j,i}$. 
In the $(i+1)^{th}$ iteration, the blocks to the left of $\mathsf{B}_{j,i}$ remain unchanged, hence $\mathsf{B}_{1,i+1},\dots,\mathsf{B}_{j-1,i+1}$ are ordered increasingly by their leaders and the leader of each of these blocks is the smallest element in that block according to $L_{i+1}$.
Let $u_1<_{L_i}\dots<_{L_i}u_x$ be the elements lying either in $\mathsf{B}_{j,i}$ or in a block to the right of $\mathsf{B}_{j,i}$, and consider the set $S=\{u_1,\dots,u_x\}$. 
In particular, $a=u_1$ since all the elements $e$ lying in $\mathsf{B}_{j,i}$ or in a block to the right of $\mathsf{B}_{j,i}$, are at least as large as their block leader according to $L_i$, and hence at least as large as $a$ according to $L_i$. 
Then the blocks $\mathsf{B}_{j,i+1},\dots,\mathsf{B}_{k,i+1}$ are obtained by applying the map $u_{y}\mapsto u_{y+1}$ (in particular $u_x \mapsto a$), and $L_{i+1}$ is obtained from $L_i$ by setting $a$ to be the last entry. 
Therefore, the relative order of the elements of $S$ as they appear from left to right in the blocks $\mathsf{B}_{j,i},\dots,\mathsf{B}_{k,i}$ according to $L_i$ is the same as the relative order that these elements appear from left to right in the blocks $\mathsf{B}_{j,i+1},\dots,\mathsf{B}_{k,i+1}$ according to $L_{i+1}$.
Thus, the statement of the lemma holds for $\mathsf{B}_{1,i+1}^{\beta_{1,i+1}}\cdots {\mathsf{B}_{k,i+1}^{\beta_{k,{i+1}}}}$, which completes the proof.
\end{proof}

\begin{lemma}\label{lem:processedOrdering}
Let $(\F,v,A)\in DCF(q,s)$, and let $(\F_1,v_1)$ be its non-distinguished part. 
{In the iterations of the steps of} Algorithm \ref{alg} {with input} $(\F_1,v_1)$, the processed elements are processed in increasing order.
\end{lemma}

\begin{proof}
We are interested in the case that at least two elements get processed.
Let $p_i$ be the element processed in the $i^{th}$ iteration ($i\geq 1$) and assume that Algorithm \ref{alg} requires at least one more iteration {of steps (1)-(6)} to terminate. 
Let $\mathsf{B}_{1,i}^{\beta_{1,i}}\cdots {\mathsf{B}_{k,i}^{\beta_{k,i}}}$ be the valued ordered chain forest obtained after the $i^{th}$ iteration.
Let $j$ be the smallest index such that $\mathsf{B}_{j,i}$ contains a non-processed element and $\beta_{j,i} > 0$, and let $a$ be the leader of $\mathsf{B}_{j,i}$. 
Note that $a$ is not a processed element since it is the smallest element in its block according to $L_i$ by Lemma \ref{lem:BlockOrder}. 
Then $p_i$ is the leader of some block of the form $\mathsf{B}_{j',i-1}$ where $j'\leq j$ since if it was the leader of some block to the right of $\mathsf{B}_{j,i-1}$, then it would not have been the left-most block satisfying the conditions in step (1) {after the $(i-1)^{th}$ iteration of the steps in} Algorithm \ref{alg} {(note that if $i-1=0$, it means that no element has been processed yet)}.
By Lemma \ref{lem:BlockOrder}, the leader of $\mathsf{B}_{j,i-1}$ is larger or equal to $p_i$ according to $L_{i-1}$ (they are equal if $j'=j$).
Since step (2) of Algorithm \ref{alg} ensures the leader of $\mathsf{B}_{j,i}$, $a$, is larger than the leader of $\mathsf{B}_{j,i-1}$ according to $L_i$, the fact that $a$ was not processed in the $i^{th}$ iteration implies that $p_i < a$. 
This completes the proof.
\end{proof}

\begin{lemma}\label{lem:ElementsProcessed}
Let $(\F,v,A)\in DCF(q,s)$, and let $(\F_1,v_1)$ be its non-distinguished part. 
After any iterative step in Algorithm \ref{alg} {with input} $(\F_1,v_1)$, each element that contributes weight to some block is a processed element. 
Conversely, each processed element either contributes weight to some block or is contained in a block with only processed elements.
\end{lemma}

\begin{proof}
We proceed by induction.
Let $\mathsf{B}_{1,i}^{\beta_{1,i}}\cdots {\mathsf{B}_{k,i}^{\beta_{k,i}}}$ be the valued ordered chain forest obtained after the $i^{th}$ iteration.
Assume that the statement is true for  $\mathsf{B}_{1,i}^{\beta_{1,i}}\cdots {\mathsf{B}_{k,i}^{\beta_{k,i}}}$. 
Here, $i=0$ corresponds to no element having been processed yet, and we restrict to the case that at least one element gets processed.
Let $j$ be the smallest index such that $\mathsf{B}_{j,i}$ contains a non-processed element and $\beta_{j,i} > 0$, and let $a$ be the leader of $\mathsf{B}_{j,i}$. 
Note that $a$ is not a processed element since it is the smallest element in its block according to $L_i$ by Lemma \ref{lem:BlockOrder}.
In the $(i+1)^{th}$ iteration, the blocks to the left of $\mathsf{B}_{j,i}$ remain unchanged, so the elements that contribute weight to one of $\mathsf{B}_{1,i+1},\dots,\mathsf{B}_{j-1,i+1}$ are processed elements (if any). 
Let $u_1<_{L_i}\dots<_{L_i}u_x$ be the elements lying either in $\mathsf{B}_{j,i}$ or in a block to the right of $\mathsf{B}_{j,i}$ and let $S=\{u_1,\dots,u_x\}$; in particular, $a=u_1$, as discussed in the proof of  Lemma \ref{lem:BlockOrder}.
Furthermore, by Lemma \ref{lem:BlockOrder}, each leader of the blocks  $\mathsf{B}_{j,i},\dots, \mathsf{B}_{k,i}$ that is not a processed element is larger than (or equal to) $a$.
Thus, if a processed element $p\in S$ in one of the blocks $\mathsf{B}_{j,i+1},\dots,\mathsf{B}_{k,i+1}$ that is not a block consisting only of processed elements, then $p$ contributes weight to this block since its leader is larger than $a$ and $p\leq a$ by Lemma \ref{lem:processedOrdering}.

Additionally, if an element contributes weight to a block {of an ordered chain forest formed after some iterative step in Algorithm \ref{alg} with corresponding ordering $L$}, then by Lemma \ref{lem:BlockOrder} this element must be a processed element; otherwise the leader of this block would not be the smallest element in the block according to $L$.

This completes the proof.
\end{proof}

If $(\F,v,A)$ is a valued $A$-distinguished ordered chain forest, with non-distinguished part $(\F_1,v_1)$ where $\sum_{\mathsf{B}\in B(\F_1)} v_1(\mathsf{B}) = q_1$, then Lemma \ref{lem:ElementsProcessed} provides a way to recover the valued ordered chain forest obtained after the $(i-1)^{th}$ iterative step in Algorithm \ref{alg} with input $(\F_1,v_1)$ from the valued ordered chain forest obtained after the $i^{th}$ iteration. 
The following Lemma describes precisely how we may do this. 

\begin{lemma}\label{lem:ReverseIteration}
Let $(\F,v,A)$ be a valued $A$-distinguished ordered chain forest of $[s]$, with non-distinguished part $(\F_1,v_1)$ where $\sum_{\mathsf{B}\in B(\F_1)} v_1(\mathsf{B}) = q_1$.
Let $(\F_{1,i},w_{1,i})= \mathsf{B}_{1,i}^{\beta_{1,i}}\cdots \mathsf{B}_{r,i}^{\beta_{r,i}}$ be {the valued ordered chain forest} obtained after the $i^{th}$ iterative step in Algorithm \ref{alg} {with input} $(\F_1,v_1)$.
Then the valued ordered
chain forest obtained after the $(i-1)^{th}$ iteration can be obtained by the following process:
    
\begin{enumerate}
\item Find the integer $j$ such that $\sum_{u=1}^{r-j}\wt(\mathsf{B}_{u,i}) + \sum_{u=r-j+1}^r |\mathsf{B}_{u,i}| + \sum_{u=1}^r \beta_{u,i} = q_1$. We note that the $j$ blocks $\mathsf{B}_{r-j+1,i},\dots,\mathsf{B}_{r,i}$ are precisely the blocks that consist only of processed elements.
        
\item Determine the set $P$ of processed elements after the $i^{th}$ iteration to consist of the elements that contribute weight to some block together with the elements of last $j$ blocks of $\mathsf{B}_{1,i}^{\beta_{1,i}}\cdots \mathsf{B}_{r,i}^{\beta_{r,}}$.
Determine $L_i$ to consist of the elements not in $P$ ordered naturally, followed by the elements in $P$ ordered naturally.
        
\item Take the largest processed element $p$, and take the leftmost block $\mathsf{B}_{v,i}$ {$(1\leq v\leq r)$} whose leader is either greater than $p$ (in the natural order) and not processed or less than (or equal to) $p$ and is processed.
        
\item  Let $u_1<_{L_i} \cdots<_{L_i} u_t$ be all the elements that appear either in $\mathsf{B}_{v,i}$ or in blocks to the right of $\mathsf{B}_{v,i}$ (note that $p=u_t$). 
Apply the map $u_x \mapsto u_{x-1}$ {where $u_1\mapsto u_{t}$ (fix $\F_1\setminus\{u_1,...,u_t\}$)}.
        
\item Form $L_{i-1}$ by modifying $L_i$ so that the ordering among the non-processed elements (including $u_t$) is natural.
        
\item For $u\neq v$, $\beta_{u,i-1} = \beta_{u,i}$ and $\beta_{v,i-1} = \beta_{v,i} + 1$.
\end{enumerate}
\end{lemma}

\begin{proof}
Note that in each iterative step in Algorithm \ref{alg}, the number of processed elements increases by 1 and the weight of some block decreases by 1. 
This implies that the total sum between the number of processed elements and the sum of the values of the blocks is fixed, equaling $q_1$ at each iteration.
Since the blocks are ordered increasingly {(by their leaders)} according to $L_i$ by Lemma \ref{lem:BlockOrder}, there is some $j$ {($0\leq j \leq r$)} such that the last $j$ blocks $\mathsf{B}_{r-j+1},\dots,\mathsf{B}_r$ consist only of processed elements, and the remaining blocks contain some element that is not processed.
By Lemma \ref{lem:ElementsProcessed}, the processed elements either lie in one of the last $j$ blocks or contribute weight to some block. 
Combining everything discussed, we have that
    \[
    \sum_{u=1}^{r-j}\wt(\mathsf{B}_{u,i}) + \sum_{u=r-j+1}^r |\mathsf{B}_{u,i}| + \sum_{u=1}^r \beta_{u,i} = q_1.
    \]

Of the processed elements after the $i^{th}$ iteration, the largest, say $p$, must be the element that was processed in the $i^{th}$ iteration by Lemma \ref{lem:processedOrdering}. 
In particular, $p$ was not a processed element after the $(i-1)^{th}$ iteration. Instead, it was the leader of some block, and this block must be $\mathsf{B}_{v,i-1}$. 
Otherwise, the left-most block of $\F_{1,i}$ whose leader is either greater than $p$ (in the natural order) and not processed or less than (or equal to) $p$ and is processed is either to the left of $\mathsf{B}_{v,i}$ (if $p$ is the leader of a block to the left of $\mathsf{B}_{v,i-1}$) or to the right of $\mathsf{B}_{v,i}$ (if $p$ is the leader of a block to the right of $\mathsf{B}_{v,i-1}$).

Then $\F_{1,i}$ was obtained by applying the map in step (2) of Algorithm \ref{alg} which is clearly reversed in step (4) of this lemma.

Finally, it is clear that $L_{i-1}$ can be formed as in step (5) and that $\beta_{v,i-1} = \beta_{v,i}+1$ and $\beta_{u,i-1} = \beta_{u,i}$ for $u\neq v$.
\end{proof}

\begin{example}
Let us follow the steps in Lemma \ref{lem:ReverseIteration} to determine the valued ordered chain forest $(\F_1,v_1)$ of weight 5 that is used as input in Algorithm \ref{alg} to obtain the valued ordered chain forest
\[
(\F_1',v_1')=[3,1,4][5,2,6][7][8]^1.
\]
The number of elements that contribute weight to some block (elements 1 and 2) added to the sum of the values of $v_1'$ over all blocks is $2+1=3$. 
Since $(\F_1,v_1)$ has weight 5, we must have exactly two processed elements that are contained in blocks that consist only of processed elements, which must be the rightmost blocks. 
In other words, the last two elements appearing in $\F_1'$, namely, 7 and 8, must be processed elements.

We now perform the reverse iterations provided by Lemma \ref{lem:ReverseIteration} of each iterative step in Algorithm \ref{alg}, which resulted in $(\F_1',v_1')$:

\begin{itemize}

\item $[3,1,4][5,2,6][7][8]^1$, $P=\{1,2,7,8\}$, $L=[3,4,5,6,1,2,7,8]$
        
\item $[3,1,4][5,2,6][8]^1[7]^1$, $P=\{1,2,7\}$, $L=[3,4,5,6,8,1,2,7]$

\item $[3,1,4][5,2,6][7]^2[8]^1$, $P=\{1,2\}$, $L=[3,4,5,6,7,8,1,2]$

\item $[2,8,3]^1[4,1,5][6]^2[7]^1$, $P=\{1\}$, $L=[2,3,4,5,6,7,8,1]$

\item $[1,7,2]^2[3,8,4][5]^2[6]^1$, $P=\emptyset$, $L=[1,2,3,4,5,6,7,8]$

\end{itemize}
Therefore, $(\F_1,v_1) = [1,7,2]^2[3,8,4][5]^2[6]^1$.
\end{example}

The fact that each iterative step in Algorithm \ref{alg} can be reversed implies that the map $\phi$ is a bijection, and therefore has an associated inverse $\phi^{-1}$.
We also emphasize that the blocks of an element $(\F,v,A) \in DCF(q,s)$ and the blocks of $\phi((\F,v,A))$ have the same lengths as they appear in order.
In particular, writing $\phi((\F,v,A)) = (\F',v',A)$, we have that $\F$ and $\F'$ have the same number of blocks and $\gamma(\F,\ell) = \gamma(\F',\ell)$ for every $0\leq \ell\leq s-1$.
Thus, we have the following result.

\begin{proposition}\label{prop:phi-bijection}
The map $\phi:DCF(q,s) \longrightarrow \phi(DCF(q,s))$ is a bijection. 
In particular, $\phi:DCF(q,s,k,\ell,m) \longrightarrow \phi(DCF(q,s,k,\ell,m))$ is a bijection, {for $1\leq m\leq k \leq s$, $0\leq \ell\leq s-1$, $0\leq q$}.
\end{proposition}

In fact, we can determine the image of $DCF(q,s,k,\ell,m)$ through $\phi$ completely, as described in the following proposition.

\begin{proposition}\label{prop:phi-image}
Let
$1\leq m\leq k \leq s,\ 0\leq \ell\leq s-1,\ 0\leq q$ be integers.
An element $(\F,v,A)$ is contained in $\phi(DCF(q,s,k,\ell,m))$ if and only if $(\F,v)= \mathsf{B}_1^{\beta_1}\cdots \mathsf{B}_r^{\beta_r} \mathsf{C}_1^{\gamma_1}\cdots \mathsf{C}_p^{\gamma_p}$ has the following properties. 
For some integer $0\leq j\leq r$,

\begin{enumerate}

\item The blocks $\mathsf{C}_1,\dots, \mathsf{C}_p$ consist only of elements in $A$, and the remaining blocks do not contain elements in $A$.
The blocks $\mathsf{C}_1,\dots, \mathsf{C}_p$ are ordered decreasingly according to their leaders, and the leader of each block $\mathsf{C}_i$ is the smallest element in the block.

\item The $j$ blocks $\mathsf{B}_{r-j+1},\dots, \mathsf{B}_r$ are naturally ordered, and the leader of each of these blocks is the smallest element in that block. 

\item The blocks $\mathsf{B}_1,\dots,\mathsf{B}_{r-j}$ are naturally ordered, and $\beta_1=\dots=\beta_{r-j} = 0$.

\item $\sum_{i=1}^{r-j} \wt(\mathsf{B}_i) + \sum_{i=j-r+1}^r (|\mathsf{B}_i| + \beta_i) + |A| + \sum_{i=1}^p\gamma_i = q$. 

\item $\F$ has $k$ blocks and $\gamma(\F,\ell) = m$.
\end{enumerate}
\end{proposition}

\begin{proof}
Let $(\mathcal{G},v,A) = \phi^{-1}((\F,v,A))$ be such that $(\mathcal{G},v,A) \in DCF(q,s,k,\ell,m)$. 
We have that 
\[
|A| + \sum_{\mathsf{B}\in B(\mathcal{G})}v(\mathsf{B})=q.
\] 
At each iterative step in Algorithm \ref{alg}, {taking as input} the non-distinguished part of $(\mathcal{G},v,A)$, an element is processed and the value of $v(\mathsf{B})$ decreases by 1 for some block $\mathsf{B}\in B(\mathcal{G})$. 
Therefore, if $P$ is the set of processed elements when Algorithm \ref{alg} terminates {with output $(\F,v)$}, we have that
\[
|P| + \sum_{i=1}^r\beta_i + |A| + \sum_{i=1}^p \gamma_i = q.
\]
Additionally, by Lemma \ref{lem:processedOrdering}, the processed elements are processed in increasing order, so the resulting ordering $L$ is given by taking the non-processed elements in increasing order followed by the processed elements in increasing order.
Since by Lemma \ref{lem:BlockOrder}, the blocks $\mathsf{B}_1,\dots,\mathsf{B}_r$ are ordered increasingly according to $L$, it follows that if $j$ is the number of blocks that consist only of processed elements, then these blocks are $\mathsf{B}_{r-j+1},\dots,\mathsf{B}_r$, and are naturally ordered by their leaders, which are the smallest elements in their respective blocks.
The remaining processed elements contribute weight to one of the blocks $\mathsf{B}_1,\dots,\mathsf{B}_{r-j}$.
We also have that $\beta_1=\cdots=\beta_{r-j} = 0$, otherwise Algorithm \ref{alg} would not have terminated since the blocks $\mathsf{B}_1,\dots,\mathsf{B}_{r-j}$ would contain elements which are not processed.

Therefore, we have that $|P| = \sum_{i=1}^{r-j}\wt(\mathsf{B}_i) + \sum_{i=j-r+1}^r |\mathsf{B}_i|$, hence
\begin{align*}
q=&|P| + \sum_{i=1}^r\beta_i + \sum_{i=1}^p \gamma_i\\
=&\sum_{i=1}^{r-j}\wt(\mathsf{B}_i) + \sum_{i=j-r+1}^r (|\mathsf{B}_i|+\beta_i) +|A| + \sum_{i=1}^p \gamma_i.
\end{align*}
Thus, $(\F,v,A)$ satisfies (2) -- (4). Moreover, (1) and (5) follow directly from the definitions.

On the other hand, if $(\F,v,A)$ satisfies (1) -- (5), then by similar lines of reasoning, if we repeatedly apply the process described in Lemma \ref{lem:ReverseIteration} to $(\F,v,A)$, we will obtain an element of $DCF(q,s,k,\ell,m)$.
\end{proof}

We are now ready to prove Conjecture \ref{conj:main}.


\section{Panhandle matroids are Ehrhart positive}\label{sec:main_proof}
This section contains a proof of Conjecture \ref{conj:main} (also printed below). We note that the validity of Conjecture \ref{conj:main} was shown in \cite{hanelyetal} to imply the Ehrhart positivity of panhandle matroids (Theorem \ref{thm:Ehrhart-paving-improved-formula}).
The proof of Conjecture \ref{conj:main} has an inclusion-exclusion flavor, utilizing the combinatorial interpretation discussed in Proposition \ref{prop:i-interpretation}. In particular, Proposition \ref{prop:i-interpretation} provides an interpretation for each of the summands in Conjecture \ref{conj:main} in terms of valued $A$-distinguished ordered chain forests. In the proof we show that most of the summand terms cancel each other out (in a nontrivial way). The only exception is the summand for $i=0$, which only partly cancels out. 
The non-canceling part is exactly the set of ordered chain forests we are after. 
The proof is presented in detail below.

\begin{repconjecture}{conj:main}
Let $1\leq m\leq k\leq s$, $q\geq 0$, and $0\leq \ell\leq s-1$ be integers. 
Then the number of naturally ordered chain forests $\F$ of $[s]$ with $k$ blocks and weight $q$ such that $\gamma(\F,\ell) = m$ is given by 
\begin{equation*}
|\CF(q,s,k,\ell,m)|=\sum_{i=0}^q(-1)^i\binom{s}{i} \Pi^{s-\ell-m}_{-i+1,s-1-\ell-i} \Pi^{\ell-(k-m)}_{s-\ell-i,s-1-i} \binom{k-1+q-i}{k-1}.
\end{equation*}
\end{repconjecture}

\begin{proof}
Let $S$ consist of the elements $(\F',v',\emptyset)=\phi((\F,v,\emptyset))$ in the image  $\phi( DCF(q,s,k,\ell,m))$, with the property that $\sum_{i=1}^k\wt(\mathsf{B}_i) = q$, where $\F' = \mathsf{B}_1\cdots \mathsf{B}_k$. 
Note that since $\sum_{i=1}^k\wt(\mathsf{B}_i) = q$, we have that $j=0$ in item (2) of Proposition \ref{prop:phi-image}, $v'$ is identically zero by item (4), and the blocks are naturally ordered by item (3). 
Therefore, if $\phi((\F,v,\emptyset)) = (\F',v',\emptyset)$ is an element of $S$, then in {the iterations in} Algorithm \ref{alg} {with input $(\F,v)$}, {there is no appearance of a block} consisting only of processed elements.

Given that $v'$ is identically zero, we may identify $(\F',v',\emptyset)$ with $\F'$. 
Since $\F'\in \CF(q,s,k,\ell,m)$ by the paragraph above, and since every element of $\CF(q,s,k,\ell,m)$ is in the image $\phi (DCF(q,s,k,\ell,m))$, $S$ is in bijection with $\CF(q,s,k,\ell,m)$.
In particular, $|S| = |\CF(q,s,k,\ell,m)|$. 
We define $T$ to be the set of elements  $(\F,v,\emptyset)\in DCF(q,s,k,\ell,m)$, such that $\phi((\F,v,\emptyset))\notin S$. 
Essentially, $T$ is the set of elements $(\F,v,\emptyset)$ where $\phi((\F,v,\emptyset))$ has at least one block consisting only of processed elements.
By Proposition \ref{prop:i-interpretation} and since $\phi$ is a bijection on $DCF(q,s,k,\ell,m)$, we have that the term {arising for $i=0$ in} the sum in Conjecture \ref{conj:main} is equal to $|\CF(q,s,k,\ell,m)| + |T|.$

Thus, we have  that $ \sum_{i=0}^q(-1)^i\binom{s}{i} \Pi^{s-\ell-m}_{-i+1,s-\ell-1-i} \Pi^{\ell-(k-m)}_{s-\ell-i,s-1-i} \binom{k-1+q-i}{k-1}$ is equal to
\begin{align*}
&\sum_{i=0}^q\sum_{A\in \binom{[s]}{i}}\sum_{(\F,v,A)\in DCF(q,s,k,\ell,m)}(-1)^{|B_A(\F)|}\\
&=|\CF(q,s,k,\ell,m)|+|T| + \sum_{i=1}^q\sum_{A\in \binom{[s]}{i}}\sum_{(\F,v,A)\in DCF(q,s,k,\ell,m)}(-1)^{|B_A(\F)|}.
\end{align*}
Hence, it suffices to show that
\[
|T| + \sum_{i=1}^q\sum_{A\in \binom{[s]}{i}}\sum_{(\F,v,A)\in DCF(q,s,k,\ell,m)}(-1)^{|B_A(\F)|}=0.
\]

\noindent To this end, we define $\mathscr{N}$ to be the set of elements $(\F,v,A)\in DCF(q,s,k,\ell,m)$ where $(-1)^{|B_A(\F)|} = -1$ (elements with a negative contribution to the sum) and $\mathscr{P}$ to be the set of elements $(\F,v,A)$ where $(-1)^{|B_A(\F)|} = 1$ and either $(\F,v,A)\in T$ or $A\neq \emptyset$ (elements with positive contribution that wish to show cancel out with elements in $\mathscr{N}$). 
In particular, each element in $\mathscr{N}$ has an odd number of blocks in its distinguished part, and each element in $\mathscr{P}$ has an even number of blocks in its distinguished part.

Now, we define a bijection $f:\mathscr{N}\longrightarrow \mathscr{P}$, which will complete the proof since the existence of $f$ implies that the elements contributing ``$+1$'' and the elements contributing ``$-1$'' cancel out in pairs in the sum $$|T| + \sum_{i=1}^q\sum_{A\in \binom{[s]}{i}}\sum_{(\F,v,A)\in DCF(q,s,k,\ell,m)}(-1)^{|B_A(\F)|}.$$
Let $(\F,v,A) = (\mathsf{B}_1^{\beta_1}\cdots \mathsf{B}_r^{\beta_r} \mathsf{C}_1^{\gamma_1}\cdots \mathsf{C}_p^{\gamma_p},A)$ be an element of $\mathscr{N}$ where the blocks $\mathsf{C}_1,\dots,\mathsf{C}_p$ are precisely the blocks that consist only of elements in $A$.
Let $\phi((\F,v,A)) = (\mathsf{B}_1'^{\beta_1'}\cdots \mathsf{B}_r'^{\beta_r'} \mathsf{C}_1^{\gamma_1}\cdots \mathsf{C}_p^{\gamma_p},A)$.
We define the map $f$ via two distinguishing  cases:

\begin{enumerate}

\item If $j = 0$ in item (2) of Proposition \ref{prop:phi-image} applied to $\phi((\F,v,A))$, or $j\geq 1$ and the leader of $\mathsf{C}_1$ is greater than the leader of $\mathsf{B}_r'$, then we define
\[
f((\F,v,A)) = \phi^{-1}((\mathsf{B}_1'^{\beta_1'}\cdots \mathsf{B}_r'^{\beta_r'} \mathsf{C}_1^{\gamma_1}\cdots \mathsf{C}_p^{\gamma_p},A\setminus \mathsf{C}_1)),
\]
where by a slight abuse of notation we associate $\mathsf{C}_1$ with the set of its elements{;} 
    
\item {Otherwise, }if $j\geq 1$ and the leader of $\mathsf{C}_1$ is less than the leader of $\mathsf{B}_r'$ (or $A=\emptyset$), then we define
\[
f((\F,v,A)) = \phi^{-1}((\mathsf{B}_1'^{\beta_1'}\cdots \mathsf{B}_r'^{\beta_r'} \mathsf{C}_1^{\gamma_1}\cdots \mathsf{C}_p^{\gamma_p},A\cup \mathsf{B}_r')).
\]
Again, we are associating $\mathsf{B}_r'$ with the set of its elements.
\end{enumerate}

In the first case, $f((\F,v,A))$ has 1 less block in its distinguished part, and in the second case, $f((\F,v,A))$ has 1 more block in its distinguished part. 
Hence, $f((\F,v,A))\in \mathscr{P}$ in both cases.
Additionally, it follows from Proposition \ref{prop:phi-image} that $(\mathsf{B}_1'^{\beta_1'}\cdots \mathsf{B}_r'^{\beta_r'} \mathsf{C}_1^{\gamma_1}\cdots \mathsf{C}_p^{\gamma_p},A\setminus \mathsf{C}_1)$ in case (1) and $(\mathsf{B}_1'^{\beta_1'}\cdots \mathsf{B}_r'^{\beta_r'} \mathsf{C}_1^{\gamma_1}\cdots \mathsf{C}_p^{\gamma_p},A\cup \mathsf{B}_r')$ in case (2) are elements of  $\phi(DCF(q,s,k,\ell,m))$, i.e., $f$ is well-defined.

We now show that the map $f$ is injective. First note that if $f((\F,v,A)) = f((\mathcal{G},w,E))$ and the distinguished parts of $(\F,v,A)$ and $(\mathcal{G},w,E)$ have the same number of blocks, then the same case {i.e., (1) or (2)} in the definition of $f$ applies to both and it follows immediately from the fact that $\phi$ is a bijection that $(\F,v,A)=(\mathcal{G},w,E)$.
Therefore, we can assume without loss of generality that the distinguished part of $(\F,v,A)$ has $p$ blocks and the distinguished part of $(\mathcal{G},w,E)$ has $p-2$ blocks and that case (1) in the definition of $f$ applies to $(\F,v,A)$ and case (2) applies to $(\mathcal{G},w,E)$. 
Then, writing 
\[
\phi((\F,v,A)) = (\mathsf{B}_1'^{\beta_1'}\cdots \mathsf{B}_r'^{\beta_r'} \mathsf{C}_1^{\gamma_1}\cdots \mathsf{C}_p^{\gamma_p},A),
\]
where the blocks $\mathsf{C}_i$ are the blocks containing elements of $A$, we have that 
\[
f((\F,v,A)) = \phi^{-1}((\mathsf{B}_1'^{\beta_1'}\cdots \mathsf{B}_r'^{\beta_r'} \mathsf{C}_1^{\gamma_1}\cdots \mathsf{C}_p^{\gamma_p},A\setminus \mathsf{C}_1)).
\]
Hence, 
\[
\phi(f((\F,v,A))) = (\mathsf{B}_1'^{\beta_1'}\cdots \mathsf{B}_r'^{\beta_r'} \mathsf{C}_1^{\gamma_1}\cdots \mathsf{C}_p^{\gamma_p},A\setminus \mathsf{C}_1),
\]
and since $\phi(f((\mathcal{G},w,E)))=\phi(f((\F,v,A)))$ we must have that
\[
\phi((\mathcal{G},w,E)) = (\mathsf{B}_1'^{\beta_1'}\cdots \mathsf{B}_r'^{\beta_r'} \mathsf{C}_1^{\gamma_1}\cdots \mathsf{C}_p^{\gamma_p},A\setminus (\mathsf{C}_1\cup \mathsf{C}_2)).
\]
By item (4) of Proposition \ref{prop:phi-image} applied to $\phi((\F,v,A))$, for $0\leq j\leq r$, we have that
\[
\sum_{i=1}^{j-r} \wt(\mathsf{B}_i') + \sum_{i=j-r+1}^r (|\mathsf{B}_i'| + \beta_i') + |A| + \sum_{i=1}^p\gamma_i = q,
\]
or equivalently,
\[
\sum_{i=1}^{j-r} \wt(\mathsf{B}_i') + \sum_{i=j-r+1}^r (|\mathsf{B}_i'| + \beta_i') + (|\mathsf{C}_1| +\gamma_1)+(|\mathsf{C}_2| + \gamma_2) + |A\setminus (\mathsf{C}_1\cup \mathsf{C}_2)| + \sum_{i=3}^p\gamma_i = q.
\]
Thus, by items (4) and (2) of Proposition \ref{prop:phi-image} applied to $\phi((\mathcal{G},w,E))$, we have that there is some $j'\geq 2$ such that the last $j'$ blocks among $\mathsf{B}_1'^{\beta_1'},\dots, \mathsf{B}_r'^{\beta_r'}, \mathsf{C}_1^{\gamma_1}, \mathsf{C}_2^{\gamma_2}$ of $\phi((\mathcal{G},w,E))$ are ordered increasingly according to their leaders. 
In particular, the leader of  $\mathsf{C}_1$ is less than the leader of $\mathsf{C}_2$. 
However, the fact that these are the first two blocks in the distinguished part of $\phi((\F,v,A))$ implies that the leader of $\mathsf{C}_1$ is greater than the leader of $\mathsf{C}_2$.
This is a contradiction and completes the proof of the injectivity of $f$.

Although it is not hard to show directly that $f$ is a surjection, we present a different argument to complete the proof. 
Let us fix suitable $s,k,l$ and allow $m$ to vary between $1$ and $k$. 
For $1\leq m \leq k$, the fact that $f$ is an injection implies that, for some $p_m\geq 0$,
\[
\sum_{i=0}^q(-1)^i\binom{s}{i} \Pi^{s-\ell-m}_{-i+1,s-\ell-1-i} \Pi^{\ell-(k-m)}_{s-\ell-i,s-1-i} \binom{k-1+q-i}{k-1} = |\CF(q,s,k,\ell,m)| + p_m.
\]
Specifically, $p_m= |\mathscr{P}| - |f(\mathscr{N})|$.
Hence, if we prove that $p_m = 0$, then $f(\mathscr{N}) = \mathscr{P}$ and it follows that $f$ is a bijection.

Summing over $m$, we have that
\begin{align*}
    &\sum_{m=0}^k\sum_{i=0}^q(-1)^i\binom{s}{i} \Pi^{s-\ell-m}_{-i+1,s-\ell-1-i} \Pi^{\ell-(k-m)}_{s-\ell-i,s-1-i} \binom{k-1+q-i}{k-1}\\
    =& 
    \sum_{i=0}^q(-1)^i\binom{s}{i} \sum_{m=0}^k\left(\Pi^{s-\ell-m}_{-i+1,s-\ell-1-i} \Pi^{\ell-(k-m)}_{s-\ell-i,s-1-i} \right)\binom{k-1+q-i}{k-1}\\
    =& \sum_{i=0}^q(-1)^i\binom{s}{i} \Pi^{s-k}_{-i+1,s-1-i}\binom{k-1+q-i}{k-1}\\
    =& \sum_{m=0}^k \left(|\CF(q,s,k,\ell,m)| + p_m\right)\\
    =&|\CF(q,s,k)| + \sum_{m=0}^k p_m.
\end{align*}
However, we know from Equation (\ref{eq:Lah}) that 
\[
|\CF(q,s,k)| = \sum_{i=0}^q(-1)^i\binom{s}{i} \Pi^{s-k}_{-i+1,s-1-i}\binom{k-1+q-i}{k-1},
\]
which implies that $\sum_{m=0}^k p_m = 0$, proving that $f$ is {indeed} a bijection.
Therefore, the result follows.
\end{proof}
This completes the proof of Conjecture \ref{conj:main}, and with it the proof of Theorem \ref{thm:PosMain}, namely, Ehrhart positivity of panhandle matroids.

\section{Ehrhart-coefficient upper-bound conjecture for paving matroids}\label{sec:upper_bound}

Recall that $U_{r,n}$ denotes the uniform matroid with rank $r$ and $n$ elements, matroid polytope $\mathcal{P}_{U_{r,n}}$ is precisely the hypersimplex $\Delta_{r,n}$. 
It is known that the Ehrhart polynomial of $\mathcal{P}_{U_{r,n}}$ has positive coefficients \cite{ferroniPositive}. 
The uniform matroid $U_{r,n}$ is the unique matroid with rank $r$ and $n$ elements and $\binom{n}{r}$ bases.
On the other hand, by a result of Dinolt \cite{DinoltMinimalMoatroids} and Murty \cite{MurtyMinimalMatroids}, there is a unique connected matroid, which we have referred to as the minimal matroid and denoted $T_{r,n}$, with rank $r$ and $n$ elements that has the minimum possible number of bases.

In what follows, for two polynomials $p(t)$ and $q(t)$ of degree $n$, $p(t)\preceq q(t)$ means that the $d^{{th}}$-coefficient of $p(t)$ is bounded above by the $d^{{th}}$-coefficient of $q(t)$, for every $0\leq d\leq n$.
In \cite{FerroniMinimal}, Ferroni conjectured the following.
\begin{conjecture}[Conjecture 1.5 in \cite{FerroniMinimal}]\label{conj:upperBound}

For a connected matroid $M$ with rank $r$ and $n$ elements,
\[
\ehr_{T_{r,n}}(t) \preceq \ehr_M(t) \preceq \ehr_{U_{r,n}}.
\]
\end{conjecture}

It is known that the lower-bound portion of the conjecture $\ehr_{T_{r,n}}(t) \preceq \ehr_M(t)$ is not true in general as Ferroni himself proved that there are matroids whose Ehrhart polynomials have negative coefficients \cite{FerroniMatroids2022}, while the Ehrhart polynomials of minimal matroids have positive coefficients \cite{FerroniMinimal}.
However, at the time of writing, there are no counterexamples to the upper-bound portion of the conjecture, $\ehr_M(t) \preceq \ehr_{U_{r,n}}$. In this section, building on the work in \cite{hanelyetal}, we show that the upper-bound portion of Conjecture \ref{conj:upperBound} holds when $M$ is a paving matroid.

{
\renewcommand{\thetheorem}{\ref{thm:UpperBound}}
\begin{theorem}
    Let $M$ be a paving matroid of rank $r$ with ground set $[n]$. Then
    \[
    \ehr_M(t) \preceq  \ehr_{U_{r,n}}(t).
    \]
\end{theorem}
}

In order to communicate a result refining Theorem \ref{thm:UpperBound} and explain how it follows from the contents of this section, we will need the following definitions and known results.

\begin{definition}

For a matroid $M$ with basis system $\B$, let $\Rel_S(\B):=\B\cup\binom{S}{r}$, for a set $S\subseteq E$ containing no basis.
We say that $S$ \emph{can be relaxed} if $\B\cup\binom{S}{r}$ is a matroid basis system. 
In this case, the resulting matroid is called the \defterm{relaxation of $M$ at $S$}, denoted by $\Rel_S(M)$.

A hyperplane $H$ of a matroid $M$ is a \defterm{stressed hyperplane} if every subset of $H$ of size $r$ is a circuit.
\end{definition}

\noindent It turns out that  stressed hyperplanes are connected to paving matroids.

\begin{proposition}{\cite[Proposition 3.16]{FerroniRelaxation}}\label{prop:stressed_hyps_paving}
A matroid is paving if and only if all its hyperplanes are stressed.
\end{proposition}

\begin{proposition}\label{prop:relaxation}\cite[Theorem~1.2]{FerroniRelaxation}
If $H$ is a stressed hyperplane of $M$, then $H$ can be relaxed. 
\end{proposition}

\noindent When $S$ is a circuit-hyperplane, the relaxation of a matroid $M$ at $S$ coincides with the common notion of a relaxation as in \cite[Section 1.5]{Oxley}, and in \cite[Definition 3.4]{hanelyetal} that, when $S$ is a stressed hyperplane, the stressed-hyperplane relaxation is recovered.
In  this section, the sets relaxed are always stressed hyperplanes, though in principle it is possible to relax other sets in matroids too. 
In \cite{hanelyetal}, a formula for the Ehrhart polynomial of a relaxation of {$M$}  was determined in terms of $\ehr_M(t)$, which is presented in the following theorem.

\begin{theorem}[Theorem 5.9 in \cite{hanelyetal}]\label{thm:Ehrhart-paving-improved-formula}
Let $M$ be a rank $r$ matroid with ground set $[n]$ and let $H$ be a stressed hyperplane such that $|H|=s\geq r$.  Then
\begin{align*}
&\ehr_{\Rel_H(M)}(t)=\ehr_{M}(t)+\frac{n-s}{(n-1)!}\binom{t-1+n-s}{n-s} \psi_{r,s,n}(t)
\end{align*}
where
\[
\psi_{r,s,n}(t) = \sum_{i=0}^{s-r}(-1)^i\binom{s}{i}\sum_{\ell=0}^{s-1}(n-2-\ell)!\ell!\binom{s-2-\ell-i+t(s-r-i+1)}{s-1-\ell}\binom{s-1-i+t(s-r-i)}{\ell}.
\]
\end{theorem}

For our purposes, the significance of Theorem \ref{thm:Ehrhart-paving-improved-formula} is that it implies that if the polynomial
\[
\binom{t-1+n-s}{n-s} \psi_{r,s,n}(t)
\]
has nonnegative coefficients for all $r\leq s\leq n-1$, then 
\[
\ehr_{M}(t) \preceq \ehr_{\Rel_H(M)}(t).
\]
Since $\binom{t-1+n-s}{n-s}$ is a polynomial with nonnegative coefficients, it suffices to show that $\psi_{r,s,n}(t)$ has nonnegative coefficients.

\begin{remark}\label{rmk:better positivity}
We have that
 \[ \frac{n-s}{(n-1)!}\binom{t-1+n-s}{n-s} \psi_{r,s,n}(t) = \ehr_{\Pan_{r,s,n}}(t)-\ehr_{U_{r-1,s}\oplus U_{1,n-s}}(t)
 \] (see the proof of Theorem 5.8 in \cite{hanelyetal} for instance). Therefore, our proof of non-negativity of $\psi_{r,s,n}(t)$, together with the factorization 
 $\ehr_{U_{r-1,s}\oplus U_{1,n-s}}(t)=\ehr_{\Delta_{r-1,s}\times \Delta_{1,n-s}}(t)$ 
imply that \[\ehr_{\Delta_{r-1,s}\times \Delta_{1,n-s}}(t)\preceq \ehr_{\Pan_{r,s,n}}(t).\]
In other words, the Ehrhart polynomial of the pandhandle matroid is bounded below (coefficient-wise) by the Ehrhart polynomial of a product of a hypersimplex with a simplex. This may be interesting in view of Ferroni's conjecture (see also Theorem \ref{thm:UpperBound}) asserting that \[
    \ehr_{\Pan_{r,s,n}}(t) \preceq  \ehr_{\Delta_{r,n}}(t).
    \]
It is also a strengthening of Theorem \ref{thm:PosMain}, since $\ehr_{\Delta_{r-1,s}\times \Delta_{1,n-s}}(t)
 =\ehr_{\Delta_{r-1,s}}(t)\ehr_{\Delta_{1,n-s}}(t)$ and
hypersimplices are Ehrhart positive \cite{ferroniPositive}. 
Even though Theorem \ref{thm:Uppermain} alone implies the non-negativity of Ehrhart polynomials of panhandle matroids, we believe Conjecture \ref{conj:main} is interesting in its own right from a purely enumerative perspective, since it enumerates a more natural set of objects. Therefore, we provide full details for the proof of Conjecture \ref{conj:main} and then discuss how the arguments can be adapted to prove the claim in Theorem \ref{thm:Uppermain}.
\end{remark}

Combining the previous result  and the following lemma about relaxations of paving matroids, which follows from Propositions \ref{prop:stressed_hyps_paving} and \ref{prop:relaxation}, we can deduce that Ferroni's upper-bound conjecture follows for the class of paving matroids from showing the nonnegativity of $\psi_{r,s,n}(t)$.

\begin{lemma}[Corollary 3.17 in \cite{FerroniRelaxation}]\label{lem:RelPaving}
Let $M$ be a paving matroid with rank $r$ on the ground set $[n]$.
The uniform matroid $U_{r,n}$ is obtained after relaxing each (stressed) hyperplane of $M$.
\end{lemma}

In \cite{hanelyetal}, it was shown that $\psi_{r,s,n}(t)$ has nonnegative coefficients if the expression given by 
\begin{equation}\label{eq:upperExpression}
\sum_{i=0}^q(-1)^i\binom{s}{i} \Pi^{s-\ell-m}_{-i,s-\ell-2-i} \Pi^{\ell-(k-m)}_{s-\ell-i,s-1-i} \binom{k-1+q-i}{k-1}
\end{equation}
is nonnegative for all $k\geq 1$, $1\leq m\leq k$, $q\geq 0$ and $0\leq\ell\leq s-1$.
Thus, we have the following.

\begin{observation}\label{obs}
If for all integers $k\geq 1$, $1\leq m\leq k$, $q\geq 0$ and $0\leq\ell\leq s-1$, the expression
\[
\sum_{i=0}^q(-1)^i\binom{s}{i} \Pi^{s-\ell-m}_{-i,s-\ell-2-i} \Pi^{\ell-(k-m)}_{s-\ell-i,s-1-i} \binom{k-1+q-i}{k-1}
\]
is nonnegative, then for any paving matroid $M$ with rank $r$ on ground set $[n]$, we have that
\[
\ehr_M(t) \preceq \ehr_{U_{r,n}}(t).
\]
\end{observation}

In what follows, we state and prove Theorem \ref{thm:Uppermain}, affirming that expression (\ref{eq:upperExpression}) is nonnegative. 
Specifically, we give a combinatorial interpretation for this expression similar to that given in Conjecture \ref{conj:main}. 
{Considering that} the above expression is so similar to the expression in Conjecture \ref{conj:main}, many of the arguments are analogous.

Another consequence of Theorem \ref{thm:Uppermain} is the following, which confirms Conjecture 6.2 in \cite{hanelyetal}.

\begin{corollary}[Conjecture 6.2 in \cite{hanelyetal}]\label{relaxation cor}
Let $M$ be a matroid with rank $r$ and let $H$ be a stressed hyperplane with $|H|\geq r$. 
If $M$ is Ehrhart positive, then so is $\Rel_H(M)$.
\end{corollary}

\begin{proof}
Theorem \ref{thm:Uppermain} and Observation \ref{obs} {imply} that $\psi_{r,s,n}(t)$ has positive coefficients for all $r\leq s\leq n-1$. 
{It then follows by Theorem \ref{thm:Ehrhart-paving-improved-formula} that} $\Rel_H(M)(t)$ has positive coefficients.
\end{proof}

In the following definition, we introduce the combinatorial objects {that} expression  (\ref{eq:upperExpression}) enumerates. 
These objects may seem unnatural compared to the naturally ordered chain forests from Conjecture \ref{conj:main}, however, these are precisely the objects that are relevant for the statement of Theorem \ref{thm:UpperBound}. 
We note that, unlike Conjecture \ref{conj:main}, the authors of \cite{hanelyetal} did not conjecture a combinatorial interpretation for expression \eqref{eq:upperExpression}. 
Rather, we are able to deduce Theorem \ref{thm:UpperBound} by applying the reasoning from the proof of Conjecture \ref{conj:main} to expression \eqref{eq:upperExpression}.

\begin{definition}\label{def:UpperInterpretation}
    Let $\CF^1(q,s,k,\ell,m)$ be the number of ordered chain forests $\F$ of $[s]$ with $k$ blocks, and $\gamma(\F,\ell) = m\geq 2$ with the following properties. 
    Writing $\F = \mathsf{B}_1\cdots \mathsf{B}_{k-1}\mathsf{C}_1$, we have{:}
    \begin{enumerate}
        \item $1\in \mathsf{C}_1$ and 1 is the leader of $\mathsf{C}_1$,
        \item $\sum_{i=1}^{k-1}\wt(\mathsf{B}_i) + |\mathsf{C}_1| = q$,
        \item The blocks $\mathsf{B}_1,\dots,\mathsf{B}_{k-1}$ are naturally ordered,
        \item The leader of the $(k-m+2)^{th}$ block is in the $(\ell+2)^{th}$ position.
    \end{enumerate}
\end{definition}
\noindent We note that the reason we require $m\geq 2$ is due to the fact that if the leader of a block is in the $(\ell+2)^{th}$ position, then some block ends at the $(\ell+1)^{th}$ position. 
Hence, there are at least 2 blocks that end after the $\ell^{th}$ position.

We prove the following theorem, which shows that the expression in question enumerates the objects {in} Definition \ref{def:UpperInterpretation}, thus proving Ferroni's Ehrhart-coefficient upper-bound conjecture for paving matroids.

\begin{theorem}\label{thm:Uppermain}

Let $1\leq m\leq k \leq s$, $q\geq 0$, and $0\leq \ell \leq s-1$ be integers. Then
\[
|\CF^1(q+1,s+1,k+1,\ell,m+1)| = \sum_{i=0}^q(-1)^i\binom{s}{i} \Pi^{s-\ell-m}_{-i,s-\ell-2-i} \Pi^{\ell-(k-m)}_{s-\ell-i,s-1-i} \binom{k-1+q-i}{k-1}.
\]
    
\end{theorem}

We now outline a proof of Theorem \ref{thm:Uppermain}. 
The proofs of the statements in this section are analogous to those in Sections \ref{sec:prelim} and \ref{sec:phi}. 
Therefore, we frequently omit the details and reference analogous results in previous sections.

\begin{definition}
Let $1\leq m\leq k \leq s$, $q\geq 0$, and $0\leq \ell \leq s-1$ be integers.
We define $DCF^1(q,s,k,\ell,m)$ to be the set of valued $A$-distinguished ordered chain forests $(\F,v,A) \in DCF(q,s,k,\ell,m)$ that have the following additional properties:
    \begin{enumerate}
        \item $1\in A$,
        \item The leader of the $(k-m+2)^{th}$ block is in the $(\ell+2)^{th}$ position,
        \item For the block $\mathsf{B}$ containing $1$, $v(\mathsf{B}) = 0$.
    \end{enumerate}
\end{definition}

The following is analogous to Proposition \ref{prop:i-interpretation}.

\begin{proposition}\label{prop:upper-i-interpretation}
Let $1\leq m\leq k \leq s$, $q\geq 0$, and $0\leq \ell \leq s-1$ be integers.
Let $T^1_i$ be the set of $i$-element subsets $A$ of $[s+1]$ such that $1\in A$, and let $Z=DCF^1(q+1,s+1,k+1,\ell,m+1)$. 
Then
    \begin{align*}
    &(-1)^i\binom{s}{i} \Pi^{s-\ell-m}_{-i,s-\ell-2-i} \Pi^{\ell-(k-m)}_{s-\ell-i,s-1-i} \binom{k-1+q-i}{k-1}
    = \sum_{A\in T^1_{i+1}} \sum_{(\F,v,A)\in Z}(-1)^{|B_A(\F)|-1}.
    \end{align*}
\end{proposition}

\begin{proof}
The proof is analogous to the proof of Proposition \ref{prop:i-interpretation}.
The quantity $\binom{s}{i}$ is the number of ways to choose an $i+1$-element subset $A$ of $[s+1]$ that contains 1.
The quantity $\Pi^{s-\ell-m}_{-i,s-\ell-2-i} \Pi^{\ell-(k-m)}_{s-\ell-i,s-1-i}$ is the sum of all products of the form $i_1\cdots i_{s-\ell-m} \cdot j_1\cdots j_{\ell-(k-m)}$ where
    \[
    -i\leq i_1<\cdots<i_{s-\ell-m}\leq s-\ell-2-i
    \]
    and
    \[
    s-\ell-i\leq j_1<\cdots<j_{\ell-(k-m)}\leq s-1-i.
    \]
In particular, among the numbers between $-i$ and $s-1-i$ (inclusive) there are precisely $k$ numbers that do not get selected to be in the product. 
As in the proof of Proposition \ref{prop:i-interpretation}, these $k$ numbers correspond to starting a block at a specific position. 
Since, in addition, the first position is always the beginning of some block, we have $k+1$ blocks in total. 
Additionally, $s-\ell-1-i$ never appears in such a product, and since this is the $(\ell+1)^{th}$ largest element among $-i,\dots,s-1-i$, this means we will always start a block at the $(\ell+2)^{th}$ position (again following the proof of Proposition \ref{prop:i-interpretation}).
Moreover, since there are $k-m$ numbers greater than $s-\ell - 1 -i$ that do not appear in the product, it follows that there are $k-m+1$ blocks (including the block starting at the first position) that come before the block starting at the $(\ell+2)^{th}$ position.
Hence, the block starting at the $(\ell+2)^{th}$ position is the $(k-m+2)^{th}$ block. 
    
The sign of such a product is given by $(-1)^{i+1 - B_A}$ where $B_A$ is the number of blocks that are filled with elements of $A$. 
Together with the $(-1)^i$ term, we have that the product $(-1)^ii_1\cdots i_{s-\ell-m} \cdot j_1\cdots j_{\ell-(k-m)}$ comes with the sign $(-1)^{|B_A|-1}$.

Finally, the quantity $\binom{k-1 +q - i}{k-1}$ is the number of ways to assign nonnegative values to the blocks that do not contain 1.
Thus, by a similar reasoning as in the proof of Proposition \ref{prop:i-interpretation}, we obtain the result.
\end{proof}

The following statements are very similar to Propositions \ref{prop:phi-bijection} and \ref{prop:phi-image}, and we omit the proofs.

\begin{proposition}\label{prop:phi-bijection-upper}
The map $\phi:DCF^1(q,s,k,\ell,m) \longrightarrow \phi(DCF^1(q,s,k,\ell,m))$ is a bijection.
\end{proposition}

\begin{proposition}\label{prop:phi0-image-upper}
An element $(\F,v,A)$ is contained in $\phi(DCF^1(q,s,k,\ell,m)$ if and only if $(\F,v)= \mathsf{B}_1^{\beta_1}\cdots \mathsf{B}_r^{\beta_r} \mathsf{C}_1^{\gamma_1}\cdots \text{{$\mathsf{C}_{p-1}^{\gamma_{p-1}}$}}\mathsf{C}_p^0$ has the following properties: 
For some integer $0\leq j\leq r$,

\begin{enumerate}

\item The blocks $\mathsf{C}_1,\dots, \mathsf{C}_p$ consist only of elements in $A$, and the remaining blocks do not contain elements in $A$. 
The blocks $\mathsf{C}_1,\dots, \mathsf{C}_p$ are ordered decreasingly according to their leaders, and the leader of each block $\mathsf{C}_i$ is the smallest element in the block.
    
\item $1\in \mathsf{C}_p$.

\item The $j$ blocks $\mathsf{B}_{r-j+1},\dots, \mathsf{B}_r$ are naturally ordered, and the leader of each of these blocks is the smallest element in that block. 

\item The blocks $\mathsf{B}_1,\dots,\mathsf{B}_{r-j}$ are naturally ordered, and $\beta_1=\dots=\beta_{r-j} = 0$.

\item $\sum_{i=1}^{r-j} \wt(\mathsf{B}_i) + \sum_{i=j-r+1}^r (|\mathsf{B}_i| + \beta_i) + |A| + \sum_{i=1}^{p-1}\gamma_i = q$. 

\item $\F$ has $k$ blocks and $\gamma(\F,\ell) = m$.

\item The leader of the $(k-m+2)^{th}$ block is in the $(\ell+2)^{th}$ position.

\end{enumerate}
\end{proposition}

We now outline the proof of Theorem \ref{thm:Uppermain}.
Again, the proof is similar to the proof of Conjecture \ref{conj:main}, so we omit most of the details.

\begin{proof}[Proof of Theorem \ref{thm:Uppermain}]
Let $Z=DCF^1(q+1,s+1,k+1,\ell,m+1)$. Let $S$ be the set of elements $\phi((\F,v,A)) = \mathsf{B}_1^{\beta_1}\cdots \mathsf{B}_k^{\beta_k} \mathsf{C}_1^0$ for some $(\F,v,A) \in DCF^1(q+1,s+1,k+1,\ell,m+1)$ where all the elements in $A$ are contained in $\mathsf{C}_1$ and $\sum_{i=1}^k\wt(\mathsf{B}_i) + |A| = q+1$.
By Proposition \ref{prop:phi-bijection-upper}, $\beta_i = 0$ for all $i$, so we may essentially identify $S$ with $\CF^1(q+1,s+1,k+1,\ell,m+1)$.
Notice that each element $(\F,v,A)$ with $\phi((\F,v,A))\in S$ contributes a ``$+1$'' to the sum
\[
\sum_{A\in T^1_{i+1}} \sum_{(\F,v,A)\in Z}(-1)^{|B_A(\F)|-1}
\]
appearing in Proposition \ref{prop:upper-i-interpretation}.

Therefore, denoting $\mathscr{P}$ to be the set of elements $(\F,v,A)\in Z$ such that $(-1)^{|B_A(\F)|-1}=1$, and $\mathscr{N}$ to be the set of elements $(\F,v,A)\in Z$ such that $(-1)^{|B_A(\F)|-1}=-1$, it suffices to find a bijection $f:\mathscr{N}\rightarrow \mathscr{P}\setminus \phi^{-1}(S)$. 
We may define $f$ in the exact same way as in the proof of Conjecture \ref{conj:main}. 
The fact that $f$ is an injection follows in the same way as in the proof of Conjecture \ref{conj:main}, and similar to the proof of Conjecture \ref{conj:main}, it is not hard to show that it is also a surjection.

This completes the proof outline.
\end{proof}


\bibliography{bibliography}
\bibliographystyle{amsalpha}
\end{document}